\documentclass{article}
\usepackage[utf8]{inputenc}
\usepackage[T1]{fontenc}
\usepackage[english]{babel}
\usepackage{tikz-cd} 
\usepackage[all,cmtip]{xy}
\usepackage{xfrac}
\usepackage{xcolor}
\usepackage[makeroom]{cancel}
\usepackage{caption}

\usepackage{hyperref}

\usepackage{amsmath}
\usepackage{amssymb}
\usepackage{amsthm} 

\newcommand{\R}{\mathbb{R}}

\newcommand{\C}{\mathbb{C}}

\newcommand{\Z}{\mathbb{Z}}

\newcommand{\N}{\mathbb{N}}

\newcommand{\T}{\mathcal{T}}

\newcommand{\M}{\mathcal{M}}
\newcommand{\Sc}{\mathcal S}

\newcommand{\MS}{\mathcal{M}\mathcal{S}}
\newcommand{\Cbi}{\mathcal{C}^\infty_b}

\newcommand{\lan}{\langle}
\newcommand{\ran}{\rangle}

\renewcommand{\S}{\mathfrak S}

\newcommand{\bep}{\boldsymbol\epsilon}
\newcommand{\bsigma}{{\boldsymbol\sigma}}
\newcommand{\bs}{{\boldsymbol s}}
\newcommand{\by}{{\boldsymbol y}}
\newcommand{\bx}{{\boldsymbol x}}
\newcommand{\bmu}{{\boldsymbol\mu}}
\newcommand{\Sig}{\Sigma_{n\times r}(\R_{\geq0})} 

\newcommand\restr[2]{{
  \left.\kern-\nulldelimiterspace 
  #1 
  \vphantom{\big|} 
  \right|_{#2} 
  }}

\newtheorem{theorem}{Theorem}[section]
\newtheorem{prop}[theorem]{Proposition} 
\newtheorem{ex}[theorem]{Example}

\theoremstyle{definition}
\newtheorem{rk}[theorem]{Remark}
\newtheorem{defn}[theorem]{Definition}
\newtheorem{defn-prop}[theorem]{Definition-Proposition}
\newtheorem{lem}[theorem]{Lemma}
\newtheorem{cor}[theorem]{Corollary}
\newtheorem{thm}[theorem]{Theorem}

\title{Pole structure of Shintani zeta functions and Newton Polytopes}
\date{}
\author{Diego L\'opez\thanks{Email address: \texttt{lopezvalenci@uni-potsdam.de}} \\\textit{Institute for Mathematics}\\\textit{University of Potsdam}\\\textit{D-14476 Potsdam, Germany.}}

\begin{document}
\maketitle
\begin{abstract}
It is known that Shintani zeta functions, which generalise multiple zeta functions, extend to meromorphic functions with poles on affine hyperplanes. We refine this result in showing that the poles lie on hyperplanes parallel to the facets of certain convex polyhedra associated to the defining matrix for the Shintani zeta function. Explicitly, the latter are the Newton polytopes of the polynomials induced by the columns of the underlying  matrix. We then prove that the coefficients of the equation which describes the hyperplanes in the canonical basis are either zero or one, similar to the poles arising when renormalising generic Feynman amplitudes. For that purpose, we introduce an algorithm to distribute weight over a graph such that the weight at each vertex satisfies a given lower bound. 
\end{abstract}

\tableofcontents
\section*{Introduction}

\addcontentsline{toc}{section}{Introduction}

The Riemann zeta function $\zeta(s)=\sum_{n\geq1}n^{-s}$ is known to be absolutely convergent whenever $\Re(s)>1$. Riemann proved in \cite{R} that $\zeta$ admits a meromorphic continuation to the whole complex plane with a simple pole at $s=1$. Several multivariable generalisations of the Riemann zeta functions have been made, for instance the multiple zeta functions (or Euler-Riemann-Zagier zeta functions) \cite{Za},\cite{Z} also called poly zeta functions \cite{C}, Shintani zeta functions \cite{GPZ1},\cite{M1}, conical zeta functions \cite{GPZ1},\cite{GPZ2},\cite{CGPZ1}, Mordell-Tornheim zeta functions \cite{M1},\cite{M2}, branched or arborified zeta functions \cite{CGPZ2, Cl1, Cl2}, etc. Meromorphic continuations of such generalisations have called the attention of numerous mathematicians \cite{AET}, \cite{M1}, \cite{M2}, \cite{Z}. 

In this document we give a precise description of the polar loci of the Shintani zeta functions which we now recall. 

\begin{defn}\label{defn:Ashintani} 
\phantom{ }
\begin{itemize}
\item Let $\Sigma_{n\times r}(\R_{\geq0})$ be the set of $n\times r$ matrices with real non-negative arguments, and with at least one positive argument in each row and in each column.
\item Given a matrix $A=\{a_{ij}\}_{1\leq i\leq n, 1\leq j\leq r}\in\Sig$ the {\bf Shintani zeta function} associated to $A$ is given by 
\begin{equation}\label{eq:shintanizeta}\zeta_A(\bs):=\sum_{m_1\geq1}\cdots\sum_{m_r\geq1}(a_{11}m_1+\dots+a_{1r}m_r)^{-s_1}\times\cdots\times(a_{n1}m_1+\dots+a_{nr}m_r)^{-s_n}.
\end{equation}
\end{itemize}
\end{defn}

The Riemann zeta function corresponds to the case when $A$ is a $1\times1$ matrix, multiple zeta functions correspond to the case when the matrix $A$ is a square lower triangular matrix with ones on and under the diagonal. Shintani zeta values, i.e. when $\bs\in\Z_{>r}^n$, span the space of conical zeta values (See \cite[Proposition 5.16]{GPZ1}).
 
 The sum on the right hand side of (\ref{eq:shintanizeta}) is absolutely convergent whenever $\Re(s_i)>r$ for every $1\leq i\leq n$ as we recall in Corollary \ref{cor:Shintani_abs_conv_domain2}. Matsumoto proved in \cite[Theorem 3]{M1} that (\ref{eq:shintanizeta}) admits a meromorphic continuation to $\C^n$ with possible linear poles located on the hyperplanes \begin{equation}\label{eq:polesMatsumoto}c_1s_1+\dots+c_ns_n=u(c_1,\dots,c_n)-l,\end{equation}
where the $c_i$ lie in $\Z_{\geq0}$, $u(c_1,\dots,c_n)$ in $\Z$, and $l$ in $\Z_{\geq0}$. His proof uses the Mellin-Barnes integration formula which involves a contour integral of the product of Gamma functions (See for instance \cite{M2}). 

The main purpose of this paper is to refine Matsumoto's result in specifying the coefficients $c_i$ in the equations of the hyperplanes describing the poles (Theorem \ref{thm:PolarlocusShintani}). We show that they are determined by normal vectors to the facets of the Newton polytopes corresponding to the product of the linear forms given by the columns of the matrix $A$. We moreover prove that the coefficients of those vectors, and therefore the coefficients $c_i$ in \eqref{eq:polesMatsumoto}, are either $0$ or $1$. The latter implies that the poles of the Shintani zeta functions generalize those of the generic Feynman amplitudes via analytic regularization, using what mathematicians call Riesz regularization in each variable. More precisely, it was shown in \cite{S,DZ} that the poles of the generic Feynman integrals using analytic regularization are of the form \begin{equation}\label{eq:Feynam_poles}s_{j_1}(s_{j_1}+s_{j_2})\cdots(s_{j_1}+\cdots+s_{j_r})=0\end{equation}
where the $s_{j_i}$ are vectors in the canonical basis of $\C^n$. Since the coefficients of the hyperplanes carrying the poles are $0$ or $1$, such poles also correspond to those of the Shintani zeta functions. However, the ones of the Shintani zeta functions are more general in the sense that they do not require the vectors to be nested as in \eqref{eq:Feynam_poles}. 

\

One of the main tools we use to study Shintani zeta functions are Newton polytopes, the definition of which we recall now. Consider a complex Laurent polynomial in $n$ variables $p:\C^n\to\C$ of the form  \begin{equation}\label{eq:notapoly}p(\bep)=\sum_{ \alpha\in \mathcal A}a_{ \alpha} \bep^{ \alpha},\end{equation} where $\mathcal A$ is a finite subset of $\Z^n$, and we have set $ \bep^{\alpha}:=\bep_1^{\alpha_1}\cdots \bep_n^{\alpha_n}$. The Newton polytope of $p$, denoted by $\Delta_p$, is the convex hull generated by $\mathcal A$. Notice that $\Delta_p$ is a compact subset of $\R^n$ since $\mathcal A$ is finite.
We illustrate this with a short  example.

\begin{ex}\label{ex:NewPol} For $p(\bep)=\epsilon_1+\epsilon_2^2+\epsilon_1\epsilon_2$ then $\mathcal{A}$ is given by $\{(1,0),(0,2),(1,1)\}$ and $\Delta_p$ is pictured below in blue color.

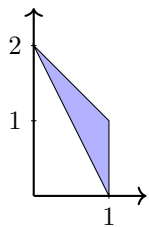
\begin{equation*}
\begin{tikzpicture}
\fill[blue!30!white] (0,2)--(1,1)--(1,0);
\draw[thick,->] (0,0) -- (1.5,0);
\draw[thick,->] (0,0) -- (0,2.5);
\foreach \x in {1}
   \draw (\x cm,1pt) -- (\x cm,-1pt) node[anchor=north] {$\x$};
\foreach \y in {1,2}
    \draw (1pt,\y cm) -- (-1pt,\y cm) node[anchor=east] {$\y$};
\draw (1,0)--(0,2)--(1,1)--(1,0);
\end{tikzpicture} 
\end{equation*}  
\end{ex}

In order to make our result more precise, let us first recall that any convex polytope in $\R^n$ as well as being the convex hull of its vertices, can also be described as the intersection of half spaces determined by its facets: 
\begin{equation}\label{eq:Newtonpolytope}
\bigcap_{k=1}^N\{{\boldsymbol \sigma}\in\R^n;\lan{\boldsymbol\mu}_k,{\boldsymbol\sigma}\ran\geq\nu_k\}.
\end{equation}
Here $N$ is the number of facets of the polytope, the $\nu_k$'s are integers, and the ${\boldsymbol\mu}_k$'s are vectors in the lattice $\Z^n$ on the inward normal direction of the facets with mutually coprime coordinates. 
We recall from \cite[Definition 2.1]{B} that a polyhedron in $\R^n$ is a set $P\subset\R^n$ defined by finitely many linear inequalitites \[P:=\{x\in \R^n: \ell_i(x)\geq \alpha_i,\,i\in I\}.\]Here $I$ is finite, $\alpha_i\in\R$ and $\ell_i:\R^n\to\R$ are linear functions. The description \eqref{eq:Newtonpolytope} shows that a polytope is in particular a polyhedron. A type of sets which will be used in the sequel is of the form $\Delta+\R^n_+$ where $\Delta$ is a polytope. Notice that $\Delta+\R^n_+$ is not a polytope since it is not bounded, however it is a polyhedron. 

\

{\bf Continuation of Example \ref{ex:NewPol}:} \textit{For} $p(\bep)=\epsilon_1+\epsilon_2^2+\epsilon_1\epsilon_2$, \textit{$\Delta_p$ can be described as the intesection of half spaces as follows.} \[ \Delta_p=\{\bsigma:\lan\bsigma,(2,1)\ran\geq2\}\cap\{\bsigma:\lan\bsigma,(-1,0)\ran\geq-1\}\cap\{\bsigma:\lan\bsigma,(-1,-1)\ran\geq-2\}.\] 

We prove in Theorem \ref{thm:PolarlocusShintani} that the coefficients $c_i$ in \eqref{eq:polesMatsumoto} are given by the coordinates of the vectors ${\boldsymbol\mu}_k$ in the canonical basis, and we express the integers $u(c_1,\cdots,c_n)$ in \eqref{eq:polesMatsumoto} in terms of the numbers $\nu_k$.
It is convenient to denote with some abuse of notation by $C_i$ a linear form in $(\R^n)^*$ associated to the column $C_i$ of the matrix, namely $C_i(\epsilon)=\lan C_i,\epsilon\ran$.

\begin{ex}
Consider the Shintani zeta associated to the matrix \[A=\begin{pmatrix}
1&0\\1&1
\end{pmatrix}.\]
Then the linear forms induced by its columns are $C_1(\epsilon_1,\epsilon_2)=\epsilon_1+\epsilon_2$ and $C_2(\epsilon_1+\epsilon_2)=\epsilon_2$. It follows from Theorem \ref{thm:PolarlocusShintani} that in this case the poles of $\zeta_A$ are parallel to the facets of $\Delta_{C_1}+\R^2_+$, $\Delta_{C_1}+\R^2_+$ and $\Delta_{C_1C_2}+\R^2_+$. The following figure represents the polyhedron $\Delta_{C_1}+\R^2_+$ and the dashed lines carry the poles parallel to its facets.

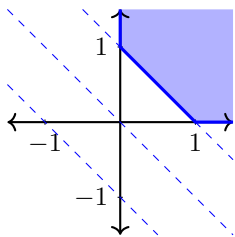
\begin{equation*}
\begin{tikzpicture}
\fill[blue!30!white] (0,1.5)--(0,1)--(1,0)--(1.5,0)--(1.5,1.5);
\draw[thick,->] (0,0) -- (1.5,0);
\draw[thick,->] (0,0) -- (0,1.5);
\draw[thick,->] (0,0) -- (-1.5,0);
\draw[thick,->] (0,0) -- (0,-1.5);
\foreach \x in {-1,1}
   \draw (\x cm,1pt) -- (\x cm,-1pt) node[anchor=north] {$\x$};
\foreach \y in {-1,1}
    \draw (1pt,\y cm) -- (-1pt,\y cm) node[anchor=east] {$\y$};
\draw[very thick, blue] (0,1.5)--(0,1)--(1,0)--(1.5,0);
\draw[dashed, blue] (-1.5,1.5)--(1.5,-1.5);
\draw[dashed, blue] (-1.5,0.5)--(0.5,-1.5);

\draw[dashed, blue] (-0.5,1.5)--(1.5,-0.5);
\end{tikzpicture} 
\end{equation*}  

Notice that, even though $\Delta_{C_1}$ has empty interior, $\Delta_{C_1}+\R^n_+$ is of maximal dimension, and its facets are the ones that induce the poles.   

\end{ex}

Specialising to  case of multiple zeta functions, our result confirms the known polar description \cite{Z}, \cite{AET}, modulo the fact that in \cite{AET} the authors proved that for the hyperplanes of the type \[s_n+s_{n-1}=2-l\] only those with $l$ even carry poles.  

\subsubsection*{Methodology} 

Our method consists of three main steps: The first step is to express the Shintani zeta function in its domain of convergence as a multiple integral, namely as the multivariable Mellin transform of a Schwartz function  on $\R^n_{+}$ which, in a neighborhood of zero, can be extended to a meromorphic function with linear poles at zero. More precisely the integrand is a product of a Schwartz function on $\R^n_{\geq0}$ and the inverse $1/C_J$ of a polynomial where $C_J(\bep)=\prod_{j\in J}\lan C_j,\bep\ran$, the $C_j$ are the columns of the matrix $A$, $\lan ,\ran$ is the canonical inner product in $\R^n$, and $J\subset [r] $. 
The second step borrows ideas from Nilsson and Passare \cite{NP} who determined the domain of convergence and the analytic continuation of the Mellin transform of a rational function. We adapt their results to the case of a product of a Schwartz function times a rational function, which we then apply to the Mellin transform obtained in step one. 
The final step is to prove that the vectors $\bmu_k$ on the inward normal direction of the polyhedra obtained in step 2 are either zero or one and to provide an easy way to derive them from the columns of the matrix $A$. For this purpose, we borrow some tools from graph theory. More precisely, we provide an algorithm to distribute weight over a graph, such that the weight at each vertex is never lower than an imposed bound.

\subsubsection*{Organization of the paper and main results}

The paper is divided in five sections. In the first one we define the space of rational functions multiplied by a Schwartz function (Definition \ref{defn:MSRplus}). Then we prove that the map $\S$ from Definition-Proposition \ref{defn-prop:SumandIntCodomain} takes values in this space. This function plays a central role in Section 3 since it is used to build an inverse Mellin transform of the Shintani zetas. The second section is devoted to the study of the Mellin transform of rapidly decreasing functions and rational functions damped by a Schwartz function. In the first part of Section \ref{sec:Mellin} we prove how the domain of convergence of the Mellin transform of a function $g$ is enlarged when multiplied by a rapidly decreasing function. In the second part of Section \ref{sec:Mellin}, based in the results from \cite{NP}, we prove (Theorem \ref{thm:thm2np-phi}) that the Mellin transform of a rational function damped by a Schwartz function admits a meromorphic extension with a very precise description of the hyperplanes carrying the poles. The third section of the paper uses the results of the previous two in order to express the Shintani zeta function associated to a matrix $A$ as the Mellin transform of $\S(A)$ (See equation \eqref{eq:sigA}) and to find a meromorphic continuation of it.
In the fourth section we show that the vectors orthogonal to the polyhedra arising from inverse Mellin transform of the Shintani zetas have coefficients 0 or 1 when expressed in the canonical basis. This refines the results of Section 3 by proving that the poles of the Shintani zetas are carried by hyperplanes of the form described in equation \eqref{eq:polesMatsumoto} where the coefficients $c_i$ are equal to 1 or 0. We also derive an easy way to calculate such vectors directly from the matrix $A$ and provide some examples. In the fifth and last section of the paper we introduce an algorithm to distribute weight over a graph such that the weight at each vertex respects a lower bound. This algorithm, which to the author's knowledge is new, is used to prove Corollary \ref{cor:algorithm} which is essential for the results in Section 4.

To conclude, we feel this paper represents a step forward in the study
of Shintani zeta functions which are meromorphic functions with linear
poles, in that it provides a geometric interpretation of the location of their poles and relates their study to graph theoretic questions. It leads to several new and exciting questions about these objects. The focus in this article is on the pole structure, as obtained in Theorems \ref{thm:PolarlocusShintani} and \ref{thm:Feynman-mu}. This calls for a full-fledged multivariable Laurent expansion of the Shintani zeta functions around the poles along the lines of  [GPZ4], a question we hope to address in future work. Such Laurent expansions would provide a consistent way to evaluate the Shintani zeta functions at the poles. Also, the links set up here with graph theory open up new paths yet to be explored.

\

{\bf Notation:} Throughout the text $[r]$ will denote the set of positive integers less or equal than $r$, namely $\{1,\cdots,r\}$, and $\{e_i\}_{i=1}^n$ are the vectors of the canonical basis of $\R^n$. We also write vectors in bold font. The symbol $\subset$ means contained, not strictly contained.  Finally $\R^n_+$ is the open set $(0,+\infty)$, while $\R_{\geq0} $ includes zero, namely $\R_{\geq0}=[0,+\infty)$.

\subsubsection*{Acknowledgments}
This paper is part of my PhD thesis funded by the DAAD (Deutscher Akademischer Austauschdienst) to which I am very grateful for their financial support. I am also thankful to Sylvie Paycha for introducing me to the topic and for the many interesting discussions around this work. Let me thank Pierre Clavier, Dominique Manchon, Kohji  Matsumoto and Yannic Vargas, for their comments on a previous version of the paper which helped me greatly improve its overall quality. Finally, I thank Kohji Matsumoto for drawing my attention to  the reference [K], as well as Thierry Champion, Nicolas Juillet, and Armand Ley for their comments and references regarding the results in Section 5.

\begin{rk}

After finishing this paper, Kohji Matsumoto pointed out to me an impressive article by Komori \cite{K} which, among other things, treats the analytic continuation of some multivariable zeta functions that generalise the Shintani zeta functions of Definition \ref{defn:Ashintani}.
Using a multidimensional contour integral representation of such zeta functions, he provides a systematic method to find the hyperplanes carrying the possible singularities. We suspect, after some computations, that our results coincide in the case of Shintani zeta functions. An advantage of our method is the geometric approach given by the Newton polynomials. Also, our method seems to provide a more computationally friendly way to calculate the location of the possible singularities. Indeed given a matrix $A\in\Sig$, we only must calculate $2^r-1$ vectors as it is illustrated in the examples of Section \ref{sec:FamilyofM}, while in the method exposed in \cite{K} one has to compute $n!$ matrices.
\end{rk}

\section{The space of rational functions damped by a Schwartz function}\label{sec:Aclassof}

We recall what a Schwartz function is: For $\mathcal O$ an unbounded connected region of $\R^n$, a function $\phi:\mathcal O\to\R$ is said to be Schwartz, denoted by $\phi\in\Sc(\mathcal O)$, if it is smooth on $\mathcal O$ and for every pair $(\boldsymbol\alpha,\boldsymbol\beta)\in\Z^{2n}_{\geq0}$ the following is satisfied: \[\lim_{\substack {\bep\to\infty\\ \bep\in\mathcal O}}\bep^{\boldsymbol\alpha}\partial^{(\boldsymbol\beta)}\phi(\bep)=0.\]
Notice that in particular $\phi\in\Sc(\R_{+}^n)$ is not necessarily bounded while $\psi\in\Sc(\R^n_{\geq0})$ is. In this paper we use the space of rational functions damped by a Schwartz function. We make this precise in the following definition.

\begin{defn}\label{defn:MSRplus}
\phantom{ }
\begin{itemize}
\item We say a function $\phi:\R^n_+\to\R$ lies in the space $\Cbi(\R^n_+)$ if, and only if there exists a polynomial $p$ in $n$ real variables such that the product $p\phi$ extends to a bounded smooth function in $\R^n_{\geq0}$ with all its derivatives bounded. 
\item We say a function $\phi:\R^n_+\to\R$ lies in the space $\MS(\R^n_+)$ if, and only if there exists a polynomial $p$ in $n$ real variables such that $p\phi$ extends to a function in $\Sc(\R^n_{\geq0})$. In other words $\MS(\R^n_+)$ is the set of rational functions damped by a Schwartz function on $\R^n_+$.
\end{itemize}
\end{defn}

 In \cite{GPZ3} the space of germs of meromorphic of functions with linear poles is used. However they implicitly use a subspace of $\MS(\R^n_+)$, namely functions of the type $\frac{\phi}{\prod L_i}$ where $\phi\in\Sc(\R^n_{\geq0})$ is analytic at zero and the $L_i$'s lie in $(\R^n)^*$. In other words, they use the space of Schwartz functions on $\R^n_+$ which can be extended,  in a neighborhood of zero, to a meromorphic function with linear poles.

\begin{prop}\label{prop:Cb-MS-algebras}
$\Cbi(\R^n_+)$ (resp. $\MS(\R^n_+)$) is an $\R$-algebra (resp. non-unital $\R$-algebra) for the pointwise product of functions.
\end{prop}

\begin{proof}
We prove that $\Cbi(\R^n_+)$ is an algebra. Let $\phi$ and $\psi$ in $\Cbi(\R^n_+)$, i.e, there are polynomials $p$ and $q$  such that $p\,\phi$ and $q\,\psi$ are bounded smooth functions on $\R^n_{\geq0}$ with all its derivatives bounded. The fact that the space of bounded smooth functions with all its derivatives bounded is an algebra implies that for $(\lambda,\mu)\in\R^2$, $\lambda\, p\,\phi+\mu\, q\,\psi$ remains in this space. Also $ p\,q\,\phi\,\psi$ remains in this space proving that $\Cbi(\R_+^n)$ is an algebra. The unit is the function identical to $1$.

We prove that $\MS(\R^n_+)$ is a non unital algebra. Let $\phi$ and $\psi$ in $\MS(\R^n_+)$, i.e, there are polynomials $p$ and $q$  such that $p\,\phi\in\Sc(\R^n_{\geq0})$ and $q\,\psi\in\Sc(\R^n_{\geq0})$. The fact that $\Sc(\R^n_{\geq0})$ is a non-unital algebra implies that \[\lambda\, p\,\phi+\mu\, q\,\psi\in \Sc(\R^n_{\geq0})\ \ \ \ \forall (\lambda,\mu)\in\R^2,\]and\[ p\,q\,\phi\,\psi\in \Sc(\R^n_{\geq0}),\]
and thus $\MS(\R_+^n)$ is a non-unital algebra, since the function identical to $1$ is not in $\MS(\R_+^n)$.
\end{proof}

\begin{ex}
Let $L(\bep)=\sum_{i=1}^na_i\epsilon_i\in(\R^n)^*$ where every $a_i>0$. Let moreover $h:\R^n_{\geq0}\to\R$ be a smooth function which its bounded together with all its derivatives. Then $\big(\bep\mapsto\frac{h(\bep)\,e^{-L(\bep)}}{L(\bep)}\big) \in\MS(\R^n_+)$.
\end{ex}

We prove a Lemma which will be useful in the sequel.

\begin{lem}\label{lem:ToddisSchwartzinR}
The Todd function defined by $z\mapsto {\rm Td}(z)=\frac{z}{e^z-1}=\frac{ze^{-z}}{1-e^{-z}}$ is a Schwartz function when restricted to $\R^n_{\geq0}$.
\end{lem}
\begin{proof}
Notice first that $\frac{1}{e^z-1}$ only has poles when $z=2\pi i k$ where $k\in\Z$, and moreover these poles are simple. Therefore ${\rm Td}$ is a meromorphic function with simple poles at $z=2\pi i k$ where $k\in\Z\setminus\{0\}$. In particular it is smooth for every $x\in\R$ and all its derivatives are bounded on the closed interval $[0,R]$ for any $R>0$. We are only left to prove that $\lim_{x\to\infty}x^\alpha{\rm Td}^{(\beta)}(x)=0$ where $\alpha$ and $\beta$ are non negative integers. 
We see first that, for $x\in\R$ \begin{equation}\label{eq:DerivativesTodd}\lim_{x\to\infty}\frac{x^pe^{-qx}}{(1-e^{-x})^r}=0\ \ \ \forall(p,r)\in\Z^2_{\geq0}\ {\rm and}\ q\geq1,\end{equation}
which follows from $\lim_{x\to\infty}\frac{1}{1-e^{-x}}=1$, and $x\mapsto e^{-qx}$ being a Schwartz function on $\R_{\geq0}$. We show now that for every $\beta\in\Z_{\geq0}$, $|{\rm Td}^{(\beta)}|$ can be expressed as a linear combination with real coefficients of fractions like the one on the left hand side of \eqref{eq:DerivativesTodd}. Indeed, it is true for $\beta=0$ since ${\rm Td}(x)=\frac{xe^{-x}}{1-e^{-x}}$. Moreover \[\frac{d}{dx}\left(\frac{x^pe^{-qx}}{(1-e^{-x})^r}\right)=\frac{px^{p-1}e^{-qx}}{(1-e^{-x})^r}-\frac{qx^pe^{-qx}}{(1-e^{-x})^r}-\frac{rx^pe^{-(q+1)x}}{(1-e^{-x})^{r+1}}\]which yields the result by induction.
\end{proof}

Using the column representation of  the matrices described in Definition \ref{defn:Ashintani}, consider the set $\mathcal C_n$ of column vectors with $n$ non negative arguments, and with at least one of them positive. Recall that, with some abuse of notation, for $C\in\mathcal C_n$ we denote also by $C$ the linear form defined by $C(\bep)=\lan C,\bep\ran$.

\begin{defn-prop}\label{defn-prop:SumandIntCodomain}The following map defined on $\mathcal C_n$ takes values in $\Cbi(\R^n_+)$.
\begin{equation}
\mathfrak S: C\mapsto\left(\R^n_+\ni\boldsymbol\epsilon\mapsto \S(C)(\bep):=\sum_{m=1}^\infty e^{-m C(\bep)}=\frac{e^{-C(\bep)}}{1-e^{-C(\bep)}}\right).\label{eq:SumColumns}
\end{equation}
We extend it linearly to $\R\mathcal C_n$ which is the real vector space freely generated by $\mathcal C_n$. (Notice that the sum in $\R\mathcal C_n$ is NOT the usual sum of vectors. For the latter $\mathfrak S$ is not a linear map.) 
\end{defn-prop}

\begin{proof}
 Notice that $C\, \S(C)=\frac{C}{e^{C}-1}={\rm Td}\circ C $ where ${\rm Td}$ is the Todd function. Lemma \ref{lem:ToddisSchwartzinR} yields the result. 
\end{proof}

We identify matrices with the tensor product of its columns. Namely, we identify $A=\{a_{i,j}\}_{1\leq i\leq n, 1\leq j\leq r}\in\Sig$ with $C_1\otimes\cdots\otimes C_r$, where $C_j=\{a_{i,j}\}_{1\leq i\leq n}$ is the $j$-th column of $A$. In order to extend the definition of $\S$ for matrices, we recall the definition and universal property of the tensor algebra. The tensor algebra of a $\R$-vector space $V$ is defined as the direct sum of the non-negative integer tensor powers of $V$, namely \begin{equation}\label{eq:tensoralgebragrading}\T(V):=\bigoplus_{k\geq0}V^{\otimes k},\end{equation}
where $V^{\otimes 0}=\R$. The product on $\T(V)$ is the concatenation of vectors, which will be denoted by $\otimes$. Notice that the concatenation product respects the grading in \eqref{eq:tensoralgebragrading}, making it a graded algebra. The unit is the inclusion $u:\R\to V^{\otimes0}\subset\T(V)$. 

We recall the universal property of the tensor algebra:

Let $V$ be a $\R$-vector space, $A$ be an $\R$-algebra and $f:V\to A$ a linear map. There exists a unique algebra morphism $\phi_f:\T(V)\to A$ such that the following diagram commutes, where $\iota$ is the natural inclusion map from $V$ to $\T(V)$.
\begin{center}\label{fig:UnivPptyTensorAlg}
\begin{tikzcd}
V\arrow{r}{\iota}\arrow{dr}{f}&\T(V)\arrow{d}{\phi_f}\\
&A
\end{tikzcd}
\captionof{figure}{Universal property of tensor algebra.}
\end{center}

By means of the universal property of the tensor algebra (Figure \ref{fig:UnivPptyTensorAlg}), $\S$ defined in equation \eqref{eq:SumColumns} extends uniquely to an algebra morphism from the tensor algebra $\mathcal T(\R \mathcal C_n)$ to $\Cbi(\R^n_+)$ (See Proposition \ref{prop:Cb-MS-algebras}). 
Namely, for $C_1\otimes\cdots\otimes C_r$

\[\S(C_1\otimes\cdots\otimes C_r)(\bep)=\prod_{j=1}^r\S(C_j)(\bep)=\sum_{\boldsymbol m\in\Z^r_+}e^{-\lan\boldsymbol m,\boldsymbol C(\bep)\ran},\]where $\boldsymbol C(\bep)=(C_1(\bep),\cdots,C_r(\bep))$.

\begin{prop}
Let $A\in\Sig$, and using the identification a matrix with the tensor product of its columns, 

\begin{equation}\label{eq:sigA}
\bep\mapsto\S(A)(\bep):=\S(C_1\otimes\cdots\otimes C_r)(\bep)=\sum_{\boldsymbol m\in\Z^r_>0}e^{-\lan A\boldsymbol m,\bep\ran}
\end{equation}
lies in $\MS(\R^n_+)$.
\end{prop}
\begin{proof}

Notice that for every column $C_j$ of $A$, $C_j(\bep)\S(C_j)(\bep)={\rm Td}(C_j(\bep))$. Since $A\in \Sig$ there is at least one non-zero element in each row, and therefore if $\bep\rightarrow\infty$, there is at least one $j'\in[r]$ such that $C_{ j'}\rightarrow\infty$. By means of Lemma \ref{lem:ToddisSchwartzinR}  $\bep\mapsto C_j(\bep)\S(C_j)(\bep)={\rm Td}(C_{j'}(\bep))$ lies in $\Sc(\R^n_{\geq0})$. Definition-Proposition \ref{defn-prop:SumandIntCodomain} implies that $\bep\mapsto C_j(\bep)\S(C_j(\bep))$ is smooth and bounded together with all its derivatives. Thus \[\Big(\prod_{j\in[r]}C_j\Big)\S(A)=\prod_{j\in[r]}C_j\S(C_j)\in\Sc(\R^n_{\geq0})\]
which yields the result.
\end{proof}

Recall that a cone $\{\R_{\geq0}a_1+\cdots+\R_{\geq0}a_m:\forall i\in[c]\, a_i\in\R^n\}\subset \R^n$ is called smooth if $m=n$ and $\{a_i\}_{i\in[n]}$ is a basis of $\Z^n$. Consider a matrix $A\in\Sigma_{n\times n}(\R_{\geq0})$ which also belongs to ${\rm SL}_n(\Z)$, then the rows of $A$ are the edges of a smooth cone $\mathcal C$ on the lattice $\Z^n_{\geq0}$, namely the columns of $A$ correspond to the vectors $a_i$ previously mentioned. In that case $\S(A)$ amounts to the discrete Laplace transform of the characteristic function of $\mathcal C\cap \Z^n$. We refer the reader to $\cite{GPZ1}$ for a more complete treatment of smooth cones and conical zeta values.

\section{The Mellin transform of classes of rapidly decreasing functions}\label{sec:Mellin}

The Mellin transform will play a central role on the sequel, since it is the main tool we use to determine meromorphic continuations. In this section we study the domain of convergence and meromorphic continuation of the Mellin transform of some specific functions. We also recall some results of \cite{NP} relative to Mellin transforms of rational functions and extend them to Mellin transforms of rational functions damped by a Schwartz function.

\begin{defn}
Let $g:\R^n_+\to\R$ be a measurable function. The Mellin tranform of $g$ is defined as \[\C^n\ni \bs\mapsto \M_{g}(\bs):=\int_{\R^n_+}\bep^{\bs-1}g(\bep)d\bep,\]
where $\bep^{\bs-1}=\epsilon_1^{s_1-1}\cdots\epsilon_n^{s_n-1}$.
\end{defn}

\begin{rk}\label{rk:Mellin_analytic}
The subset of $\C^n$ for which the above integral is defined depends on the function $g$. However as a consequence of Morera's theorem \cite{Mo}, if the integral is convergent on an open set $\mathcal{O}\subset\C^n$ it is also holomorphic on $\mathcal{O}$ (See for instance, \cite[Theorem 5.4]{SS}).
\end{rk}

\subsection{Domain of convergence of the Mellin transform of rapidly decreasing functions}

We introduce a class of functions, whose behavior at infinity ensures a non empty domain of convergence of its Mellin transform. 

\begin{defn}
We call a function $\phi:\R^n_{\geq0}\to\R$ rapidly decreasing if for every $\alpha\in\Z^n$, \[\lim_{\bep\to\infty}\bep^\alpha \phi(\bep)=0.\]
\end{defn}

Notice that this definition differs from that of Schwartz functions, in that here it is not required the function to be smooth and its derivatives to be rapidly decreasing. For the rest of this section we adopt the notation $\C^n_+:=\{\bs\in\C^n:\Re(\bs)>0\}$.
\begin{prop}\label{prop:ExtendedDomainofMellin}
Let $g:\R^n_+\to\C$ be a function, the Mellin transform of which 
\begin{equation*}
\M_g(\bs)=\int_{\R^n_+}\bep^{\bs-1}g(\bep)d\bep
\end{equation*}
is absolutely convergent for values of $\bs$ in some non-empty set $D_g\subset\C^n$. For any given bounded, measurable function  $\phi:\R^n_{\geq0}\to\R$ which is rapidly decreasing, the Mellin transform of $g\phi$ given by
\begin{equation*}
\M_{g\phi}(\bs)=\int_{\R^n_+}\bep^{\bs-1}g(\bep)\phi(\bep)d\bep
\end{equation*}
is absolutely convergent on the set $D_g+\C_+^n$ where the sum is understood as the usual sum of subsets of a vector space (Minkowski sum). It is moreover analytic on the interior of $D_g+\C_+^n$.
\end{prop}
\begin{proof}
Let $\bs$ in $D_{g}+\C_+^n$, then there is a $\bs_0\in D_g$ and $\boldsymbol \alpha\in\C_+^n$ such that $\bs=\bs_0+\boldsymbol \alpha$. Let $\bar B_R(0)$ be a closed ball of $\R^n_{\geq0}$ centered at 0 with radius $R>0$ and set $\bsigma:=\Re(\bs)$, $\bsigma_0:=\Re(\bs_0)$, and $\boldsymbol a=\Re(\boldsymbol \alpha)$, then formally
\begin{equation}\label{eq:splittingMgphi}
|\M_{g\phi}(\bs)|\leq\int_{\bar B_R(0)}\bep^{\bsigma_0+\boldsymbol a-1}|g(\bep)\phi(\bep)|d\bep+\int_{\R^n_+\setminus \bar B_R(0)}\bep^{\bsigma_0+\boldsymbol a-1}|g(\bep)\phi(\bep)|d\bep.
\end{equation}
It is therefore enough to prove the convergence of the two integrals on the right hand side of \eqref{eq:splittingMgphi} to prove that $\M_{g\phi}(\bs)$ is convergent.
The facts that $\phi$ is bounded, $\boldsymbol a$ lies in $\R_+^n$, and $\bar B_R(0)$ is compact, imply that there is a $M_a>0$ such that for every $\bep\in \bar B_R(0)$ the inequality $\bep^{\boldsymbol a}|\phi(\bep)|<M_a$ holds, and therefore 
\[\int_{\bar B_R(0)}\bep^{\bsigma_0+\boldsymbol a-1}|g(\bep)\phi(\bep)|d\bep\leq M_a\int_{\bar B_R(0)}\bep^{\bsigma_0-1}|g(\bep)|d\bep\leq M_a\int_{\R^n_+}\bep^{\bsigma_0-1}|g(\bep)|d\bep\]
which is convergent by assumption.

For the second integral on the right hand side of \eqref{eq:splittingMgphi}, since $\phi$ is rapidly decreasing, for $R$ big enough $\bep^{\boldsymbol a}|\phi(\bep)|<1$ for every $\bep\in\R^n_+\setminus \bar B_R(0)$ and thus
\[\int_{\R^n_+\setminus \bar B_R(0)}\bep^{\bsigma_0+\boldsymbol a-1}|g(\bep)\phi(\bep)|d\bep\leq\int_{\R^n_+\setminus \bar B_R(0)}\bep^{\bsigma_0-1}|g(\bep)|d\bep\leq \int_{\R^n_+}\bep^{\bsigma_0-1}|g(\bep)|d\bep,\]which is again convergent by assumption. Therefore both integrals on the right hand side of \eqref{eq:splittingMgphi} are convergent implying the expected result. The analiticity follows from Remark \ref{rk:Mellin_analytic}. 
\end{proof}

\begin{rk}\label{rk:MellinSchwartz}Notice that the case $g(\bep)=1$ is not covered by the former proposition since $M_1(\bs)$ does not converge for any $\bs\in\C^n$. However it is easy to see that for $\phi$ bounded, measurable and rapidly decreasing, $M_\phi(\bs)$ converges and is analytic for $\Re(\bs)>1$, by this we mean $\Re(s_i)>1$ for every $i\in[n]$. Indeed
\begin{equation}\label{eq:rkmellinschw}
\M_\phi(\bs)=\int_{B_1(0)\cap\R_+^n}\bep^{\bs-1}\phi(\bep)d\bep+\int_{\R_+^n\setminus B_1(0)}\bep^{\bs-1}\phi(\bep)d\bep,\end{equation}
for $\Re(\bs)>1$ the first integral on \eqref{eq:rkmellinschw} converges since the integrand is bounded and $B_1(0)$ is compact. For the second integral it is enough to realize that $|\bep^{\bs-1}\phi(\bep)|\leq C\, \frac{1}{|\bep|^{n+1}}$ for some $C>0$.
 \end{rk}

The following Definition and Theorem were stated and proved by Nilsson and Passare \cite{NP}, and will be of use in the sequel.

\begin{defn}\label{defn:truncated}
\phantom{ }
\begin{itemize}
\item Let $p$ be a complex Laurent polynomial on $n$ variables and $\Gamma$ a face of its Newton polytope $\Delta_p$. The {\bf truncated polynomial} $p_\Gamma$ associated to the facet $\Gamma$ is the sum of the monomials of $p$ whose exponents lie in $\Gamma$.
\item We say a polynomial is {\bf completely non-vanishing} on a subset $\mathcal{O}\subset \C^n$ if neither it, nor its truncated polynomials, vanish in $\mathcal{O}$.
\end{itemize}
\end{defn}

\begin{thm}\label{thm:thm1np}\cite{NP}
If a polynomial $p$ is completely non-vanishing on $\R^n_+$, then the integral \begin{equation*}\label{eq:mellintrratfuncnp}\M_{1/p}(\bs)=\int_{\R^n_+}\frac{\bep^{\bs-1}}{p(\bep)}d\bep
\end{equation*}
converges absolutely and defines an analytic function $\bs\mapsto \M_{1/p}(\bs)$ on the tube domain \[D_{1/p}:=\{\bs\in\C^n:\Re(\bs)=\bsigma\in {\rm int}(\Delta_p)\}.\]
\end{thm}

We need an extension of this theorem whose proof follows very closely the one in \cite{NP}.

\begin{thm}\label{cor:mellinprodschwratz}
Let $\phi:\R^n_{\geq 0}\to\R$ be a measurable, bounded, rapidly decreasing function, and $p$ a polynomial completely non-vanishing on $\R^n_+$, then the Mellin transform of $\phi/p$ 
\begin{equation*}\label{eq:mellintrratfunc}\M_{\phi/p}(\bs)=\int_{\R^n_+}\frac{\bep^{\bs-1}\phi(\bep)}{p(\bep)}\, d\bep,\end{equation*}
converges absolutely and defines an analytic function  $\bs\mapsto \M_{\phi/p}(\bs)$ on the tube domain $\Re(\bs)=\bsigma\in \Delta_{p}+\R^n_+$.
\end{thm}

\begin{proof}
Whenever $\Delta_p$ has no empty interior, the statement is a direct consequence of Theorem \ref{thm:thm1np} and Proposition \ref{prop:ExtendedDomainofMellin}. For the case ${\rm int}(\Delta_p)=\emptyset$, we adapt the proof of \cite[Theorem 1]{NP}. 

Consider the change of variable $\epsilon_i\mapsto e^{x_i}$, then \[\M_{\phi/p}(s)=\int_{\R^n}\frac{e^{\lan \bs,\bx\ran}\phi(e^\bx)}{p(e^{\bx})}d\bx.\]
Similar to Remark \ref{rk:MellinSchwartz}, it is then enough to prove that there is a bounded, measurable rapidly decreasing function $\psi$ such that \begin{equation}\label{eq:boundedintegrand}
\frac{e^{\lan\bsigma,\bx\ran}|\phi(e^\bx)|}{|p(e^\bx)|}\leq \psi(\bx)
\end{equation}
for every $\bx\in\R^n\setminus K$ where $K$ is a compact set. We prove this by induction over the dimension $n$. For $n=1$, since ${\rm dim}(\Delta_p)=0$ then $p(e^x)=a_\alpha e^{\alpha x}$ where $\Delta_p=\alpha\in\Z$. Let $M>0$ be such that $|\phi(e^x)|<M$ for every $x$. For negative values of $x$  
\[\frac{e^{\sigma x}|\phi(e^x)|}{|a_\alpha e^{\alpha x}|}\leq\frac{M}{|a_\alpha|}e^{-(\sigma-\alpha)|x|}.\]Since $e^{-(\sigma-\alpha)|x|}$ is bounded, measurable and rapidly decreasing for $x\leq 0$, this yields the result. On the other hand, $|e^{(\sigma-\alpha)x}\phi(e^x)|$ is rapidly decreasing for $x>0$, thus our claim is true for $n=1$.

For the inductive step assume that inequality in \eqref{eq:boundedintegrand} holds for dimensions smaller than $n$, and consider a polynomial in $n$ variables $p(\bep)=\sum_{\alpha\in\mathcal{A}}a_\alpha \bep^\alpha$, where $\mathcal A\subset\Z^n$, and with ${\rm dim}(\Delta_p)<n$. Let $\bsigma\in\Delta_p+\R^n_+$, then in particular $\bsigma\notin\Delta_p$ (because $0\notin \R_+$). We build a family of cones, the union of which covers $\R^n\setminus K$ where $K$ is a compact set. Recall that for a set $X\subset\R^n$, ${\rm conv}(X)$ refers to the convex hull generated by the elements in $X$. 
\begin{itemize}
\item For every face $\Gamma$ of $\Delta_p+\R^n_+$, choose $\bsigma_\Gamma\in{\rm int}(\Gamma)$. Define $\Delta_\Gamma={\rm conv}((\mathcal A\setminus\Gamma)\cup\{\bsigma_\Gamma,\bsigma\})$. Consider then the cone \[\widetilde C_\Gamma:=\{\bx\in\R^n:\,(\forall\boldsymbol\xi\in\Delta_\Gamma)\, \lan\boldsymbol\xi-\bsigma_\Gamma,\bx-\bsigma_\Gamma\ran\leq0\}.\]
\item Consider the polytope $\Delta_{\bsigma} :={\rm conv}(\mathcal A\cup\{\bsigma\}).$ Then define the cone \[\widetilde C_{\bsigma} :=\{\bx\in\R^n:\,(\forall\boldsymbol\xi\in\Delta_{\bsigma})\, \lan\boldsymbol\xi-\bsigma,\bx-\bsigma\ran\leq0\}.\]
\end{itemize}
Notice that $\R^n\setminus \left(\bigcup_{\Gamma}\widetilde C_\Gamma\cup \widetilde C_{\bsigma}\right)$ is a bounded set, where $\Gamma$ takes values on the set of faces of $\Delta_p+\R^n_+$. Even more, consider for every $\Gamma$ (resp. for $\bsigma$) a slightly smaller closed, convex cone  $C_\Gamma$ (resp. $C_{\bsigma}$) contained in the interior of $\widetilde C_\Gamma$ (resp. $\widetilde C_{\bsigma}$) with vertex in $\bsigma_\Gamma$ (resp. $\bsigma$) such that $\R^n\setminus \left(\bigcup_{\Gamma} C_\Gamma\cup  C_{\bsigma}\right)$ is still a bounded set. It is then enough to prove the estimate \eqref{eq:boundedintegrand} for every $\bx$ in $C_\Gamma$ and for every $\bx$ in  $C_{\bsigma}$ with norm $|\bx|$ chosen sufficiently large. 

Fix a face $\Gamma$ of $\Delta_p+\R^n$ and consider $\bx\in\C_\Gamma$. With a slight abuse of notation we call $p_\Gamma$  the sum of the monomials of $p$, the exponents of which lie in $\Gamma$. Notice that this differs from the truncated polynomials in Definition \ref{defn:truncated} in that here $\Gamma$ is not necessarily a face of $\Delta_p$ but rather one of $\Delta_p+\R^n_+$. Let $q_\Gamma=p-p_\Gamma$. Then
\begin{equation}\label{eq:splitthm1NP}
\frac{e^{\lan\bsigma,\bx}\ran\phi(e^\bx)}{p(e^\bx)}=\frac{e^{\lan\bsigma-\bsigma_\Gamma,\bx\ran}}{(\phi(e^\bx))^{-1}e^{-\lan\bsigma_\Gamma,\bx\ran}p_\Gamma(e^\bx)+(\phi(e^\bx))^{-1}e^{-\lan\bsigma_\Gamma,\bx\ran}q_\Gamma(e^\bx)}.
\end{equation}
Let us now determine adequate bounds for the numerator and for each of the sumands in the denominator of the fraction on the right hand side of \eqref{eq:splitthm1NP}.
\begin{itemize}
\item Bound for $e^{\lan\bsigma-\bsigma_\Gamma,\bx\ran}$: Set $\by:=\bx-\bsigma_\Gamma$ and $k:={\rm min}\{\langle\bsigma_\Gamma-\bsigma,\by\ran\, : |\by|=1,\, \bsigma_\Gamma+\by=\bx\in C_\Gamma\}$. By construction of $C_\Gamma$ we have $k>0$. Then \begin{align*}
e^{\langle\bsigma_\Gamma-\bsigma,\bx\rangle}&=e^{\lan\bsigma_\Gamma-\bsigma,\bsigma_\Gamma\ran}e^{\lan\bsigma_\Gamma-\bsigma,\frac{\by}{|\by|}\ran|\by|}\\
&\geq e^{\lan\bsigma_\Gamma-\bsigma,\bsigma_\Gamma\ran}e^{k|\by|}\\
&\geq e^{\lan\bsigma_\Gamma-\bsigma,\bsigma_\Gamma\ran}e^{k(|\bx|-|\bsigma_\Gamma|)}\\
&\geq c_1e^{k|\bx|}
\end{align*}
where we have set $c_1:=e^{\lan\bsigma_\Gamma-\bsigma,\bsigma_\Gamma\ran}e^{-k|\bsigma_\Gamma|}>0$. Setting $\psi(\bx)=(c_1e^{k|\bx|})^{-1}$ implies that $|e^{\lan\bsigma-\bsigma_\Gamma,\bx\ran}|<\psi(x)$ where $\psi$ is a bounded, measurable, rapidly decreasing function.
\end{itemize}
We now find a constant $c>0$ such that the denominator of \eqref{eq:splitthm1NP} $(\phi(e^\bx))^{-1}e^{-\lan\bsigma_\Gamma,\bx\ran}p_\Gamma(e^\bx)\ +\ (\phi(e^\bx))^{-1}e^{-\lan\bsigma_\Gamma,\bx\ran}q_\Gamma(e^\bx))>c$ for $|\bx|\in C_\Gamma$ large enough.
\begin{itemize}
\item Bound for $(\phi(e^\bx))^{-1}e^{-\lan\bsigma_\Gamma,\bx\ran}p_\Gamma(e^\bx)$: Notice that ${\rm dim}(\Delta_{p_\Gamma})=m<n$. Therefore, a change of coordinates, the transformation matrix of which has determinant 1 (in dimension 2 and 3 is equivalent to a rotation of coordinates) can be made, such that in the new coordinates $\bx'$ the polytope $\Delta_{p_\Gamma}$ is in a (possibly affine) subspace parallel to the span of the first $m$ coordinates $x'_1,\cdots,x'_m$, and the remaining $n-m$ coordinates are orthogonal to $\Delta_{p_\Gamma}$.  We write $\bx'=(\bx'_1,\bx'_2)$ where $\bx'_1$ are the first $m$ coordinates  and $\bx'_2$ are the last $n-m$ coordinates, we write similarly $\bsigma_\Gamma=(\bsigma_{\Gamma1},\bsigma_{\Gamma2})$. Then $p_\Gamma(e^{\bx'})=e^{\lan \sigma_{\Gamma2},\bx'_2\ran}p_{\Gamma}(e^{\bx'_1})$. Indeed, since all the exponents of the polynomial $p_\Gamma$ (seen as vectors on $\Z^n$) lie in $\Delta_{p_\Gamma}$, then their last $n-m$ components in the coordinates $\bx'$ are equal to $\bsigma_{\Gamma_2}$, because $\Delta_{p_\Gamma}$ is constant in those components. It follows that

\[|(\phi(e^\bx))^{-1}e^{-\lan\bsigma_\Gamma,\bx\ran}p_\Gamma(e^\bx)|= |(\phi(e^{\bx}))^{-1}||e^{-\lan\bsigma_{\Gamma1},\bx'_1\ran}p_\Gamma(e^{\bx'_1})|.\]
Since the right hand side depends on $m<n$ variables, by means of the induction hypothesis, there is a constant $c_2>0$ such that $|(\phi(e^\bx))^{-1}e^{-\lan\bsigma_\Gamma,\bx\ran}p_\Gamma(e^\bx)|>c_2$ for $|\bx|$ large enough.   

\item Bound for $(\phi(e^\bx))^{-1}e^{-\lan\bsigma_\Gamma,\bx\ran}q_\Gamma(e^\bx)$: We now proceed to prove that for $|\bx|$ large enough $|\phi^{-1}(e^\bx)e^{-\lan\bsigma_\Gamma,\bx\ran}q_\Gamma(e^\bx)|<c_2/2$. Recall that \[q_\Gamma(e^{\bx})=\sum_{\alpha\in\mathcal A\setminus\Gamma}a_\alpha e^{\lan\alpha,\bx\ran}.\]
Since $\alpha\in\Delta_\Gamma$ and $C_\Gamma$ is closed, the following constant exist and is positive
\[k_\alpha:={\rm min}\{\lan\bsigma_\Gamma-\alpha,\by\ran:\, |\by|=1,\ \bx=\bsigma_\Gamma+\by\in C_\Gamma \}.\] 
Hence
\begin{align*}
|a_\alpha e^{\lan\alpha,\bx\ran}e^{-\lan\bsigma_\Gamma,\bx\ran}|&=|a_\alpha e^{\lan\alpha,\bsigma_\Gamma\ran-|\bsigma_\Gamma|^2}e^{-\lan\bsigma_\Gamma-\alpha,\frac{\by}{|\by|}\ran|\by|}|\\
&\leq  |a_\alpha e^{\lan\alpha,\bsigma_\Gamma\ran-|\bsigma_\Gamma|^2}e^{-k_\alpha|\by|}|.
\end{align*}
The last term tends to zero as $|\by|$ tends to infinity. Therefore, for $|\bx|$ large enough $|(\phi(e^\bx))^{-1}e^{-\lan\bsigma_\Gamma,\bx\ran}q_\Gamma(e^\bx)|<\frac{c_2}{2}$ as expected.
\end{itemize}

It follows that for $\bx\in C_\Gamma$ with $|\bx|$ large enough, the estimate \eqref{eq:boundedintegrand} is satisfied. We are only left to prove the same estimate for $\bx\in C_\bsigma$. For that purpose notice that for $\bx$ with large enough norm, $\bx\in C_\bsigma$ implies that $\bx\notin \R^n_{\leq0}$. It follows then that $e^\bx\mapsto\phi(e^{\bx})$ is rapidly decreasing for $\bx\in C_{\bsigma}$. On the other hand, since $p$ is a polynomial completely non-vanishing in $\R^n_+$, then $C_\bsigma\ni\bx\mapsto| p(e^\bx)|$ is bounded from below by a positive constant. Therefore \[C_\bsigma\ni\bs\mapsto\frac{e^{\lan\bsigma,\bx\ran}|\phi(e^\bx)|}{|p(e^\bx)|}\]
is a bounded, measurable rapidly decreasing function, which completes the proof.
\end{proof}

\subsection{Meromorphic continuation of the Mellin transform}

In this paragraph we build a meromorphic continuation for the Mellin transform of rational functions damped by a Schwartz function (See Definition \ref{defn:MSRplus}). For that purpose we first recall a result from \cite{NP} which we then adapt to our context.

\begin{thm}\cite[Theorem 2]{NP}
Let $p$ be a completely non-vanishing polynomial on the positive orthant $\R^n_+$ such that its Newton polytope $\Delta_p$ is of full dimension. Then the Mellin transform of $1/p$
\begin{equation*}\M_{1/p}(\bs):=\int_{\R^n_+}\frac{\bep^{\bs-1} }{p(\bep)}d\bep\end{equation*}
admits a meromorphic continuation of the form  \begin{equation*}\M_{1/p}(\bs)=\Phi(\bs)\prod_{k=i}^{N}\Gamma(\lan\boldsymbol\mu_k,\bs\ran-\nu_k),\end{equation*}
where $\Phi$ is an entire function, and $N$, $\boldsymbol\mu_k$ and $\nu_k$ are as in \eqref{eq:Newtonpolytope}.
\end{thm}

We need a slightly more general form of the previous theorem, the proof of which follows closely the one in \cite{NP} with a small change to take into account a Schwartz function.

Similar to \eqref{eq:Newtonpolytope}, there are integers $\nu_k$, and vectors $\boldsymbol\mu_k$ which lie in $\Z_{\geq0}^n$ on the inward normal direction of the facets of $\Delta_p+\R_+^n$ with mutually coprime coordinates, such that \begin{equation}\label{eq:NewtonpoltoInfty}
\Delta_p+\R_+^n:=\bigcap_{k=1}^N\{\bsigma\in\R^n;\lan\boldsymbol\mu_k,\bsigma\ran\geq\nu_k\}.
\end{equation}
\begin{thm}\label{thm:thm2np-phi}
Let $p$ be a completely non-vanishing polynomial on the positive orthant $\R^n_+$ and $\phi\in\Sc(\R_{\geq0}^n)$. Then the Mellin transform 
\begin{equation}\label{eq:intthm2}\M_{\phi/p}(\bs)=\int_{\R^n_+}\frac{\bep^\bs \ \phi(\bep)}{p(\bep)}\frac{d\bep}{\bep}\end{equation}
admits a meromorphic continuation of the form  \begin{equation*}\label{eq:merextthm2}\M_{\phi/p}(\bs)=\Phi(\bs)\prod_{k=i}^{N}\Gamma(\lan\boldsymbol\mu_k,\bs\ran-\nu_k),\end{equation*}
where $\Phi$ is an entire function, and $N$, $\boldsymbol\mu_k$ and $\nu_k$ are as in \eqref{eq:NewtonpoltoInfty}.
\end{thm}
\begin{proof}
This proofs closely follows that of Nilsson and Passare. We use here the following notation introduced by them: For a given $\boldsymbol\gamma\in\Z^N$
\[\Delta(\boldsymbol\gamma):=\bigcap_{k=1}^N\{\bsigma\in\R^n:\lan\boldsymbol\mu_k,\bsigma\ran\geq\gamma_k\},\]
in particular $\Delta(\boldsymbol\nu)=\Delta_p+\R_+^n$, where $\boldsymbol\nu$ is a vector, the components of which are the $\nu_k$ as in equation \eqref{eq:NewtonpoltoInfty}. 

We claim that for every $\boldsymbol m\in\Z_{\geq0}^N$, there are functions $\phi_{\boldsymbol m,i}\in\Sc(\R_{\geq0}^n)$ and polynomials $q_{\boldsymbol m,i}$ whose Newton polytopes satisfy $\Delta_{q_{\boldsymbol m,i}}\subset \Delta(|\boldsymbol m|\boldsymbol\nu+\boldsymbol m)$, such that 
\begin{equation}\label{eq:claimthmNP2}
\M_{\phi/p}(\bs)=\frac{1}{\prod_{j=1}^Nu_{\boldsymbol m,j}(\bs)}\left(\sum_{i=1}^{N_{\boldsymbol m}}\int_{\R^n_+}\frac{\boldsymbol\epsilon^\bs q_{\boldsymbol m,i}(\boldsymbol\epsilon)\ \phi_{\boldsymbol m,i}(\boldsymbol\epsilon)}{p(\boldsymbol\epsilon)^{1+|\boldsymbol m|}}\frac{d\boldsymbol\epsilon}{\boldsymbol\epsilon}\right)
\end{equation}
Here $N_{\boldsymbol m}\in\N$, $|\boldsymbol m|=m_1+\dots+m_N$, and $u_{\boldsymbol m,j}(\bs)=\prod_{l=0}^{m_j-1}(\lan\boldsymbol\mu_j,\bs\ran-\nu_j+l)$.
We prove this by induction over $|\boldsymbol m|$. Let $\boldsymbol m=e_k$: for $\lambda>1$, introducing the change of coordinates $(\epsilon_1,\dots ,\epsilon_n)\mapsto(\lambda^{\mu_{k1}}\epsilon_1,\dots,\lambda^{\mu_{kn}}\epsilon_n)$ in \eqref{eq:intthm2} we obtain
\begin{equation}\label{eq:bassestepcoordinatechange}
\M_{\phi/p}(\boldsymbol s)=\lambda^{\lan\boldsymbol \mu_k,\boldsymbol s\ran-\nu_k}\int_{\R^n_+}\frac{\boldsymbol \epsilon^{\boldsymbol s} }{\lambda^{-\nu_k}p(\lambda^{\boldsymbol \mu_k}\boldsymbol \epsilon)}\phi(\lambda^{\boldsymbol \mu_k}\boldsymbol \epsilon)\frac{d\boldsymbol \epsilon}{\boldsymbol \epsilon}.
\end{equation}

From differentiating \eqref{eq:bassestepcoordinatechange} with respect to $\lambda$ and evaluating at $\lambda=1$, it follows that
\begin{equation*}\label{eq:afterdiff}
0=(\lan\boldsymbol \mu_k,\boldsymbol s\ran-\nu_k)\M_{\phi/p}(\boldsymbol s)-\int_{\R^n_+}\frac{\boldsymbol \epsilon^{\boldsymbol s} q_{e_k}(\boldsymbol \epsilon)}{p^2(\boldsymbol \epsilon)}\phi(\boldsymbol \epsilon)\frac{d\boldsymbol \epsilon}{\boldsymbol \epsilon}+\int_{\R^n_+}\frac{\boldsymbol \epsilon^{\boldsymbol s} }{p(\boldsymbol \epsilon)}\left(\sum_{l=1}^n\epsilon_l\mu_{k_l}\frac{\partial\phi}{\partial{\epsilon_l}}\Big|_{\boldsymbol \epsilon}\right)\frac{d\boldsymbol \epsilon}{\boldsymbol \epsilon},
\end{equation*}
where we have set 
\[q_{e_k}(\bep)=\frac{d}{d\lambda}\left(\lambda^{-\nu_k}p(\lambda^{\boldsymbol \mu_k}\bep)\right)\Big|_{\lambda=1}.\]

That implies\begin{equation*}
\M_{\phi/p}(\boldsymbol s)=\frac{1}{\lan\boldsymbol \mu_k,\boldsymbol s\ran-\nu_k}\left(\int_{\R^n_+}\frac{\boldsymbol \epsilon^{\boldsymbol s} q_{e_k}(\boldsymbol \epsilon)}{p^2(\boldsymbol \epsilon)}\phi(\boldsymbol \epsilon)\frac{d\boldsymbol \epsilon}{\boldsymbol \epsilon}-\sum_{l=1}^n\mu_{k_l}\int_{\R^n_+}\frac{\boldsymbol \epsilon^{\boldsymbol s} p(\boldsymbol \epsilon)\epsilon_l}{p^2(\boldsymbol \epsilon)}\frac{\partial\phi}{\partial{\epsilon_l}}\Big|_{\boldsymbol \epsilon}\frac{d\boldsymbol \epsilon}{\boldsymbol \epsilon}\right).
\end{equation*}

Since $\frac{\partial\phi}{\partial{\boldsymbol \epsilon_l}}\in\Sc(\R_{\geq0}^n)$, we are only left to prove $\Delta_{q_{e_k}}\subset\Delta(\boldsymbol \nu+e_k)$ and $\Delta_{p\, \epsilon_l}\subset\Delta(\boldsymbol \nu+e_k)$. For the first inclusion, let $\Gamma_{e_k}$ be the facet of $\Delta_p+\R_+^n$ contained on the hyperplane $\lan\boldsymbol \mu_k,\bsigma\ran=\nu_k$. Notice that $q_{e_k}$ contains only the monomials from $p$ whose exponents do not lie on the facet $\Gamma_{e_k}$, therefore $\Delta_{q_{e_k}}\subset\Delta(\boldsymbol \nu+e_k)$. We prove the second inclusion, namely $\Delta_{p\, \epsilon_l}\subset\Delta(\boldsymbol \nu+e_k)$: recall that the Newton polytope of the product of two polynomials is equal to the sum of the Newton polytopes of each polynomial, thus $\Delta_{p\, \epsilon_l}$ is $\Delta_{p}$ translated by the vector $e_l$. Then if $\mu_{k_l}\neq0$, $\bsigma\in\Delta_{p\,\epsilon_l}$ implies that for every $j\in[N]$, $\lan\bsigma-e_l,\boldsymbol \mu_j\ran\geq\nu_j$ or equivalently $\lan\bsigma,\boldsymbol \mu_j\ran\geq\nu_j+\mu_{j_l}\geq \nu_j+\delta_{j,k}$ where the last inequality is a consequence of $\boldsymbol \mu_k\in\Z_{\geq0}^n$ and $\mu_{k_l}\neq0$. This proves that $\Delta_{p(\epsilon)\epsilon_i}\subset\Delta(\boldsymbol \nu+e_k)$ and our claim for the case $|\boldsymbol m|=1$.

For the inductive step assume that \eqref{eq:claimthmNP2} holds for a vector $\boldsymbol m$. We prove that it also holds for $\boldsymbol m':=\boldsymbol m+e_k$. Consider on each of the integrals on the right hand side of \eqref{eq:claimthmNP2} the change of variables $(\epsilon_1,\dots ,\epsilon_n)\mapsto(\lambda^{\mu_{k1}}\epsilon_1,\dots,\lambda^{\mu_{kn}}\epsilon_n)$. From differentiating with respect to $\lambda$ and making $\lambda =1$, it follows that
\begin{equation*}\footnotesize{
0=(\lan\boldsymbol \mu_k,\boldsymbol s\ran-\nu_k+m_k)\int_{\R^n_+}\frac{\boldsymbol \epsilon^{\bs} q_{\boldsymbol m,i}(\boldsymbol \epsilon)\ \phi_{\boldsymbol m,i}(\boldsymbol \epsilon)}{p(\boldsymbol \epsilon)^{1+|\boldsymbol m|}}\frac{d\boldsymbol \epsilon}{\boldsymbol \epsilon}-\int_{\R_+^n}\frac{\boldsymbol \epsilon^\bs q_{\boldsymbol m',i}(\boldsymbol \epsilon)\phi_{\boldsymbol m,i}(\boldsymbol \epsilon)}{p(\boldsymbol \epsilon)^{2+|\boldsymbol m|}}\frac{d\boldsymbol \epsilon}{\boldsymbol \epsilon}+\int_{\R^n_+}\frac{\boldsymbol \epsilon^\bs q_{\boldsymbol m,i}(\boldsymbol \epsilon)}{p(\boldsymbol \epsilon)^{1+|\boldsymbol m|}}\left(\sum_{l=1}^n\epsilon_l \mu_{k_l}\frac{\partial\phi_{\boldsymbol m,i}}{\partial{\epsilon_p}}\Big|_{\boldsymbol \epsilon}\right)\frac{d\boldsymbol \epsilon}{\boldsymbol \epsilon}},
\end{equation*}
implying \begin{equation*}\footnotesize{
\int_{\R^n_+}\frac{\boldsymbol \epsilon^\bs q_{\boldsymbol m,i}(\boldsymbol \epsilon)\ \phi_{\boldsymbol m,i}(\boldsymbol \epsilon)}{p(\boldsymbol \epsilon)^{1+|\boldsymbol m|}}\frac{d\boldsymbol \epsilon}{\boldsymbol \epsilon}=\frac{1}{\lan\boldsymbol \mu_k,\bs\ran-\nu_k+m_k}\left(\int_{\R_+^n}\frac{\boldsymbol \epsilon^\bs q_{\boldsymbol m',i}(\boldsymbol \epsilon)\phi_{\boldsymbol m,i}(\boldsymbol \epsilon)}{p(\boldsymbol \epsilon)^{2+|\boldsymbol m|}}\frac{d\boldsymbol \epsilon}{\boldsymbol \epsilon}-\sum_{l=1}^n\mu_{k_l}\int_{\R^n_+}\frac{\boldsymbol \epsilon^\bs q_{\boldsymbol m,i}(\boldsymbol \epsilon)p(\boldsymbol \epsilon)\epsilon_l}{p(\boldsymbol \epsilon)^{2+|\boldsymbol m|}}\frac{\partial\phi_{\boldsymbol m,i}}{\partial{\epsilon_l}}\Big|_{\boldsymbol \epsilon}\frac{d\boldsymbol \epsilon}{\boldsymbol \epsilon}\right)}
\end{equation*}
where $q_{\boldsymbol m',i}(\boldsymbol \epsilon)=(1+|\boldsymbol m|)q_{e_k}(\boldsymbol \epsilon)q_{\boldsymbol m,i}(\boldsymbol \epsilon)-p(\boldsymbol \epsilon)\tilde{q}_{\boldsymbol m,i}(\boldsymbol \epsilon)$, and \[\tilde{q}_{\boldsymbol m, i}(\boldsymbol \epsilon)=\frac{d}{d\lambda}\left(\lambda^{-|\boldsymbol m|\nu_k-m_k}q_{\boldsymbol m,i}(\lambda^{\boldsymbol \mu_k}\boldsymbol \epsilon)\right)\Big|_{\lambda=1}.\]

To prove our claim, we are only left to show that $\Delta_{q_{\boldsymbol m',i}}\subset\Delta(|\boldsymbol m'|\boldsymbol \nu+\boldsymbol m')$ and $\Delta_{q_{\boldsymbol m,i} p\, \epsilon_l}\subset\Delta(|\boldsymbol m'|\boldsymbol \nu+\boldsymbol m')$. For the first inclusion, notice that $\Delta_{q_{e_k}q_{\boldsymbol m,i}}\subset\Delta(\boldsymbol \nu+e_k)+\Delta(|\boldsymbol m|\boldsymbol \nu+\boldsymbol m)\subset\Delta(|\boldsymbol m'|\boldsymbol \nu+\boldsymbol m')$ where we have use the general inclusion $\Delta(\boldsymbol \alpha)+\Delta(\boldsymbol \beta)\subset\Delta(\boldsymbol \alpha+\boldsymbol \beta)$. Moreover, since any of the monomials in the polynomial $\tilde q_{\boldsymbol m,i}$ has exponents on the hyperplane $\lan\boldsymbol \mu_k,\bsigma\ran=|\boldsymbol m|\nu_k+m_k$, it follows that $\Delta_{p\,\tilde q_{\boldsymbol m,i}}\subset\Delta(\boldsymbol \nu)+\Delta(|\boldsymbol m|\boldsymbol \nu+\boldsymbol m+e_k)\subset\Delta(|\boldsymbol m'|\boldsymbol \nu+\boldsymbol m')$. This yields the first inclusion. For the second inclusion, recall that if $\mu_{k_l}\neq0$, $\Delta_{p\, \epsilon_l}\subset\Delta(\boldsymbol \nu+e_k)$, then \[\Delta_{q_{\boldsymbol m,i} p\, \epsilon_l}=\Delta_{q_{\boldsymbol m,i}}+\Delta_{p\, \epsilon_l}\subset\Delta(|\boldsymbol m|\boldsymbol \nu+\boldsymbol m)+\Delta(\boldsymbol \nu+e_k)\Delta(|\boldsymbol m'|\boldsymbol \nu+\boldsymbol m')\]
proving the second inclusion. Doing the same procedure for each of the integrals on the right hand side of \eqref{eq:claimthmNP2} we obtain 

\begin{equation*}
\M_{\phi/p}(\bs)=\frac{1}{\prod_{j=1}^Nu_{\boldsymbol m',j}(\bs)}\left(\sum_{i=1}^{N_{\boldsymbol m'}}\int_{\R^n_+}\frac{\boldsymbol \epsilon^\bs q_{\boldsymbol m',i}(\boldsymbol \epsilon)\ \phi_{\boldsymbol m',i}(\boldsymbol \epsilon)}{p(\epsilon)^{1+|\boldsymbol m'|}}\frac{d\boldsymbol \epsilon}{\boldsymbol \epsilon}\right),
\end{equation*}
as expected.

We prove now that each of the domains of convergence of each of integrals on the right hand side of \eqref{eq:claimthmNP2} contains $\Delta(\boldsymbol \nu-\boldsymbol m)$. Fix $ i$ in $[ N_{\boldsymbol m}]$, by means of Theorem \ref{cor:mellinprodschwratz}, the $i$-th integral in \eqref{eq:claimthmNP2} converges on \begin{equation}\label{eq:intersectionofDomains}
\bigcap_{\boldsymbol \tau\in\Delta_{q_{\boldsymbol m,i}}}\left((1+|\boldsymbol m|)\Delta(\boldsymbol \nu)-\boldsymbol \tau\right).
\end{equation}
We show then that $\Delta(\boldsymbol \nu-\boldsymbol m)$ is a subset of \eqref{eq:intersectionofDomains}. Indeed, if $\bsigma\in\Delta(\boldsymbol \nu-\boldsymbol m)$, then for every $j$ in $[ N]$
\[\lan\bsigma,\boldsymbol \mu_j\ran\geq\nu_j-m_j,\]
moreover, the inclusion $\Delta_{q_{\boldsymbol m,i}}\subset \Delta(|\boldsymbol m|\boldsymbol \nu+\boldsymbol m)$ yields \[\lan\boldsymbol \tau,\boldsymbol \mu_j\ran\geq|\boldsymbol m|\nu_j+m_j,\ \ \ \forall \tau\in\Delta_{q_{\boldsymbol m,i}}.\]
Both inequalities imply that
\begin{equation*}
\lan\bsigma+\boldsymbol \tau,\boldsymbol \mu_j\ran\geq(1+|\boldsymbol m|)\nu_j,
\end{equation*}
and thus $\bsigma$ lies on the intersection \eqref{eq:intersectionofDomains}.

Inside the domain $\{\bs\in\C^n: \boldsymbol \sigma \in\Delta(\boldsymbol \nu-\boldsymbol m)+\R^n_+\}$ the only poles of $\M_{\phi/p}$ are given by $u_{\boldsymbol m,j}(\bs)=0$, and are simple, which coincide with the poles of the product $\prod_{k}\Gamma(\lan\bs,\boldsymbol \mu_k\ran-\nu_k)$. Therefore, by the removable singularities Theorem, $\Phi(\boldsymbol s)=\M_{\phi/p}(\boldsymbol s)/(\prod_{k}\Gamma(\lan\bs,\boldsymbol \mu_k\ran-\nu_k))$ is an entire function which yields the expected result.
\end{proof}

\section{Polar structure of Shintani zeta functions}

In this paragraph we prove the main result announced in the introduction which gives a precise description of the polar loci of the Shintani zeta functions $\zeta_A$ associated to a matrix $A$, as in Definition \ref{defn:Ashintani}. For this purpose, we bring together the results of Sections \ref{sec:Aclassof} and \ref{sec:Mellin}. We recall that for a matrix $A\in\Sig$ with columns are $C_1,\dots,C_r$, $\S(A)=\S(C_1\otimes\cdots\otimes C_r)$ (See \eqref{eq:sigA}). We also recall that for a given column vector $C_j$ we also denote, with some abuse of notation, by $C_j$ the linear form defined by $C_j(\bep):=\lan\bep,C_j\ran$.

\begin{prop}\label{prop:Shintani_abs_conv_domain}
Let $A\in\Sig$, then for $\bs$ in the domain of convergence of $\zeta_A$ \[\zeta_A(\bs)\Gamma(\bs)=\M_{\S(A)}(\bs).\]
Moreover, for every $A\in\Sig$, $\zeta_A$ is absolutely convergent whenever $\M_{\S(A)}(\bs)$  is absolutely convergent.
\end{prop}
\begin{proof}The statement follows from:
\begin{align*}
\M_{\S(A)}(\bs)=\M_{\S(C_1\otimes\cdots\otimes C_r)}(\bs)&=\int_{\R_{+}^n}\bep^{\bs-1}\left(\prod_{j=1}^r\sum_{\boldsymbol m\in\Z_{\geq1}^r}e^{- m_jC_j (\bep)}\right)d\bep\\
&=\int_{\R_{+}^n}\bep^{\bs-1}\left(\sum_{\boldsymbol m\in\Z_{\geq1}^r}e^{-\langle A^t\bep, \boldsymbol m\rangle}\right)d\bep\\
&=\sum_{\boldsymbol m\in\Z_{\geq1}^r}\int_{\R_{+}^n}\bep^{\bs-1}e^{-\langle\bep, A\boldsymbol m\rangle}d\bep\\
&=\sum_{\boldsymbol m\in\Z_{\geq1}^r}\prod_{i=1}^n\int_{\R_{+}}\epsilon_i^{s_i-1}e^{-L_i(\boldsymbol m)\epsilon_i}d\epsilon_i\\
&=\sum_{\boldsymbol m\in\Z_{\geq1}^r}\prod_{i=1}^nL_i(\boldsymbol m)^{-s_i}\int_{\R_{+}}t_i^{s_i-1}e^{-t_i}dt_i=\zeta_{A}(\bs)\Gamma(\bs).
\end{align*}
Here $\Gamma(\bs)=\Gamma(s_1)\cdots\Gamma(s_n)$, and the $L_i$'s are the lines of the matrix $A$. As for the columns, we denote with some abuse of notation $L_i(\boldsymbol m)=\lan\boldsymbol m,L_i\ran$. In the third line Fubini's theorem is used since $e^{-\langle\bep, A\boldsymbol m\rangle}$ is positive. The statement of absolute convergence follows also from Fubini's theorem.

\end{proof}

We focus on the study of $\M_{\S(A)}$ in order to build a meromorphic continuation for $\zeta_A$. For that purpose we prove two lemmas.

\begin{lem}\label{lem:hbounded} The map $x\mapsto e^{-x}h(x)$ lies in $\Sc(\R_{\geq0})$, where $h:\R_{\geq0} \to \R$  is defined as $h(x)=\frac{{\rm Td}(x)e^x-1}{x}$ or equivalently \[\frac{{\rm Td}(x)}{x}=e^{-x}\left(\frac{1}{x}-h(x)\right).\]
 \end{lem}
 \begin{proof}
 Indeed, $x\mapsto e^{-x}h(x)=\frac{{\rm Td}(x)}{x}-\frac{e^{-x}}{x}$, thus it lies in $\Sc(\R_+)$ as a consequence of Lemma \ref{lem:ToddisSchwartzinR}. Moreover, in a neighborhood of zero \begin{align*}e^{-x}h(x)&=\frac{{\rm Td}(x)}{x}-\frac{e^{-x}}{x}\\
&=\sum_{n=-1}^\infty B_{n+1}\frac{x^{n}}{(n+1)!}-\sum_{n=-1}^\infty \frac{x^{n}}{(n+1)!}\\
&=\sum_{n=0}^\infty (B_{n+1}-1)\frac{x^{n}}{(n+1)!}.
\end{align*}
Here $B_n$ are the Bernoulli numbers, and we used the fact that $B_0=1$. Therefore $x\mapsto e^{-x}h(x)$ is analytic at zero with $h(0)=-\frac{3}{2}$, and in particular $x\mapsto e^{-x}h(x)\in\Sc(\R_{\geq0})$.
\end{proof}

For the rest of this paragraph we set $\phi(x):=e^{-x}\in\Sc(R_{\geq0})$ in order to simplify the notation. 

\begin{lem}\label{lem:MellinasSum}
For $A\in\Sig$, and $J\subset [r]$, we have that $\phi(C_{[r]}(\bep))h(C_J(\bep))\in\Sc(\R_{\geq0}^n)$, where we used the compact notation $h(C_J(\bep)):=\prod_{j\in J}h(C_j(\bep))$ and $\phi(C_{[r]}(\bep)):=\prod_{j=1}^r\phi(C_j(\bep))$. Furthermore, we have 
\[\M_{\S(A)}(\bs)=\M_{\S(C_1\otimes\dots\otimes C_r)}(\bs)=\sum_{I\sqcup J=[r]}\int_{\R_+^n}\bep^{\bs-1}\frac{ \phi(C_{[r]}(\bep))h(C_J(\bep))}{C_I(\bep)}d\bep.\] 

\end{lem}

\begin{proof}
Fix $J\subset[r]$, we prove that $\phi(C_{[r]}(\bep))h(C_J(\bep))\in\Sc(\R_{\geq0}^n)$. Since $A\in\Sig$, there is at least a non-zero element in each row. The latter implies that for every path in which $\bep$ tends to $\infty$, there is at least a $j\in[r]$ such that $C_j(\bep)\to\infty$. 
Assume $j\in J$, then Lemma \ref{lem:hbounded} implies that 
$\phi(C_{J}(\bep))h(C_J(\bep)){\longrightarrow}0$ as $\bep\to\infty$. Since $\phi(C_{[r]\setminus J})$ is bounded, this yields the result. Assume otherwise $j\in[r]\setminus J$, then $\phi(C_{[r]\setminus J}(\bep))\longrightarrow0$ as $\bep\to\infty$. By Lemma \ref{lem:hbounded} $\phi(C_{J})h(C_J)$ is bounded and therefore $\phi(C_{[r]}(\bep))h(C_J(\bep))\in\Sc(\R_{\geq0}^n)$.

On the other hand, since $\S(C_j)=\frac{{\rm Td}(C_j)}{C_j}$ (see Definition-Proposition \ref{defn-prop:SumandIntCodomain}), with the notations of Lemma \ref{lem:hbounded}, it follows that:
\begin{align*}
\M_{\S(C_1\otimes\dots\otimes C_r)}(\bs)&=\int_{\R_+^n}\bep^{\bs-1}\S(C_1\otimes\cdots\otimes C_r)(\bep)d\bep\\
&=\int_{\R_+^n}\bep^{\bs-1}\prod_{j=1}^r\S(C_j(\bep))d\bep\\
&=\int_{\R_+^n}\bep^{\bs-1}\prod_{j=1}^r\left(\frac{e^{-C_j(\bep)}}{C_j(\bep)}-e^{-C_j(\bep)}h(C_j(\bep))\right)d\bep\\
&=\sum_{I\sqcup J=[r]}\int_{\R_+^n}\bep^{\bs-1}\frac{ \phi(C_{[r]}(\bep))h(C_J(\bep))}{C_I(\bep)}d\bep,
\end{align*}
which proves the statement.
\end{proof}

We show what is the domain of convergence for a sum of the type shown in \eqref{eq:shintanizeta}.
\begin{prop}\label{prop:Shintani_conv_domain}
Let $A\in\Sig$, the sum on the right hand side of the Shintani zeta function $\zeta_A$ as defined in \eqref{eq:shintanizeta} is absolutely convergent for 
\begin{equation}\label{eq:intersection_1}\bsigma\in\bigcap_{I\subset[r]}\left(\Delta_{C_I}+\R_+^n\right),\end{equation}
where $\bsigma=\Re(\bs)$.
\end{prop}
\begin{proof}
From Proposition \ref{prop:Shintani_abs_conv_domain}, it is enough to prove that $\M_{\S(A)}$ is absolutely convergent for $\bsigma$ lying in the intersection on the right hand side of \eqref{eq:intersection_1}. The result follows from Lemma \ref{lem:MellinasSum} and Theorem \ref{cor:mellinprodschwratz}. 
\end{proof}

As a corollary, we recall a well known result which we prove for the sake of completeness.

\begin{cor}\label{cor:Shintani_abs_conv_domain2}
Let $A\in\Sig$, the sum on the right hand side of the Shintani zeta function $\zeta_A$ as defined in \eqref{eq:shintanizeta} is absolutely convergent for $\Re{s_i}>r$ for every $1\leq i\leq r$.
\end{cor}
\begin{proof}
By means of Proposition \ref{prop:Shintani_conv_domain}, we have to show that any $\bs$ such that $\Re\bs=\bsigma$ lies in the interior of the intersection \begin{equation}\label{neq:intersection}\bigcap_{I\subset[r]}\left(\Delta_{C_I}+\R_+^n\right).\end{equation}

Since $\Delta_{C_i}\subset\{(x_1,\dots,x_n)\in\R^n_{\geq0}:\, x_1+\cdots+x_n=1\}$, then for any $I\subset[r]$, $\Delta_{C_I}\subset\{(x_1,\dots,x_n)\in\R^n_{\geq0}:\, x_1+\cdots+x_n=|I|\}$.  We prove that
\[\Delta_{C_I}+\R^n_+\supset\{(x_1,\cdots,x_n)\in\R_+^n:\forall i\in[n]\,, x_i>|I|\}.\]
Indeed, for any $\bsigma \in \{(x_1,\cdots,x_n)\in\R_{\geq0}^n:\forall i\in[n]\,, x_i>|I|\}$ and every $\bsigma_0\in\{(x_1,\dots,x_n)\in\R^n_{\geq0}:\, x_1+\cdots+x_n=|I|\}$, it is straightforward to see that $\bsigma-\bsigma_0\in\R^n_+$.

 In particular, the set $\{\bsigma\in\R^n: \sigma_i>r\}$ is a subset of the intersection in \eqref{neq:intersection} which proves the result.
\end{proof}

As mentioned in the introduction, our aim is to build a meromorphic continuation of \eqref{eq:shintanizeta} to the whole space $\C^n$ and to give a precise description of its poles. In doing so, we refine a result by Matsumoto \cite{M2} that we have recalled in the introduction. Our result follows from Theorem \ref{thm:thm2np-phi}.

\begin{thm}\label{thm:PolarlocusShintani}
Let $A\in\Sig$ and write $C_J(\bep)=\prod_{j\in J}C_{j}(\bep)$ where $J\subset[r]$ and $C_j$ the $j$-th  column of $A$. The Shintani zeta function \[\zeta_A(\bs)=\sum_{m_1\geq1}\cdots\sum_{m_r\geq1}(a_{11}m_1+\dots+a_{1r}m_r)^{-s_1}\times\cdots\times(a_{n1}m_1+\dots+a_{nr}m_r)^{-s_n}\] admits a meromorphic continuation to $\C^n$ with possible simple poles located on the hyperplanes \[\langle\boldsymbol\mu_k^J,\bs\rangle=\nu_k^J-l.\]
Here $\boldsymbol\mu_k^J$ are the vectors in $\Z^n$ on the inward normal direction of the facets of $\Delta_{C_J}+\R^n_+$ with mutually coprime coordinates, $\nu_k^J$ are integers as described in \eqref{eq:NewtonpoltoInfty}, and $l\in\Z_{\geq0}$. 

Moreover if $\boldsymbol\mu_k^J=e_i$ for some vector of the canonical basis $e_i$, then only the hyperplanes where $0\leq l\leq \nu_k^J-1$ carry poles.
\end{thm}
\begin{proof}
By means of Proposition \ref{prop:Shintani_abs_conv_domain} and Lemma \ref{lem:MellinasSum}, it is enough build an meromorphic continuation and describe the polar locus for the integrals \[\int_{\R_+^n}\bep^{\bs-1}\frac{ \phi(C_{[r]}(\bep))h(C_{[r]\setminus J}(\bep))}{C_J(\bep)}d\bep\]
for every $J\subset[r]$. It follows from Theorem \ref{thm:thm2np-phi} that for $J$ fixed the aforementioned integral admits a meromorphic continuation to the whole space $\C^n$ where the poles are the same as those of the functions $ \Gamma(\langle\boldsymbol\mu_k^J,\bs\rangle-\nu_k^J)$. This proves the first part of the Theorem.

Moreover, if there is a $\boldsymbol\mu_k^J=e_i$ for any $e_i$ by means of the removable singularities theorem, the function $\zeta_A(\bs)=\frac{1}{\Gamma(\bs)}\M_{\S(C_1\otimes\cdots\otimes C_r)}(\bs)$ has no poles on the hyperplanes $\langle\boldsymbol\mu_k^J,\bs\rangle=-l$ where $l\geq\nu_k^J$ which yields the result.

\end{proof}

Using the polar description of the Shintani zetas provided in the previous Theorem, we can describe the change of the polar structure under some linear transformations of $\bs$. Notice that this amounts to a transformation of the polyhedra which induce the polar structure.

\begin{cor}
Let $A$ be a matrix in $\Sig$ and $B$ an $n\times n$ real matrix, then the function $\bs\mapsto\zeta_A(B\bs)$ admits a meromorphic continuation to the space $\C^n$ with possible simple poles located on the hyperplanes $\langle B^t\boldsymbol\mu_k^J,\bs\rangle=\nu_k^J-l$. Here $B^t$ is the transpose of $B$, and $\boldsymbol\mu_k^J$, $\nu_k^J$ and $l$ are as in Theorem \ref{thm:PolarlocusShintani}.
\end{cor}

\begin{proof}
It is immediate from Theorem \ref{thm:PolarlocusShintani}. 
\end{proof}

The following corollary shows how the pole structure of a Shintani zeta function $\zeta_A$ associated to a matrix $A$ only depends on the positions of the zeros in the matrix. It does not depend on the values of the other arguments as long as they are positive. In particular, it is enough to consider the matrices $S\in \Sig$ with arguments $0$ or $1$ to study all possible configurations of the poles for Shintani zeta functions. This is in sharp contrast with the convergent case, since the values of $\zeta_A$ outside the poles depend strongly on the values of the coefficients of $A$. 

\begin{cor}\label{cor:A-squeleton}
Let $A\in\Sig$ and consider a new matrix $\bar A$ defined by $\bar a_{ij}=0$ if $a_{ij}=0$, and $\bar{a}_{ij}=1$ if $a_{ij}\neq0$. Then $\zeta_A$ and $\zeta_{\bar A}$ have the same pole structure.
\end{cor} 
\begin{proof}
Let $C_j$ (resp. $\bar C_j$) be the columns of the matrix $A$ (resp. $\bar A$). Once again with some abuse of notation, we denote by the same symbol the linear forms $C_j(\bep)=\lan\bep,C_j\ran$ and $\bar C_j(\bep)=\lan\bep,\bar C_j\ran$.  Since $C_j(\epsilon)=\sum_{i=1}^na_{ij}\epsilon_i$ and $\bar C_j(\epsilon)=\sum_{i=1}^n\bar a_{ij}\epsilon_i$ where $a_{ij}=0$ if, and only if $\bar a_{ij}=0$, it follows from the definition of Newton polytope (see \eqref{eq:notapoly}) that $\Delta_{C_j}=\Delta_{\bar C_j}$. Since the Newton polytope of the product of polynomials is equal to the Minkowski sum of their Newton polytopes, it follows that $\Delta_{C_J}=\Delta_{\bar C_J}$ for any $J\subset[r]$. Theorem \ref{thm:PolarlocusShintani} then yields the result.
 
\end{proof}

\section{Family of meromorphic germs spanned by the Shintani zeta functions}\label{sec:FamilyofM}

In this paragraph, we study the space of meromorphic germs at zero spanned by the Shintani zeta functions. For that purpose, we give a better description of the vectors $\mu_K^J$ obtained in Theorem \ref{thm:PolarlocusShintani}.  The main result of this section is Theorem \ref{thm:Feynman-mu} which states that the arguments of the vectors $\mu_K^J$ are either $0$ or $1$ if represented in the canonical basis of $\R^n$. This proves that the poles at zero are similar to the ones of generic Feynman amplitudes studied in $\cite{S, DZ}$. However, the family of meromorphic germs spanned by those containing Shintani zeta functions is bigger that the one spanned by those containing generic Feynman amplitudes, since the latter are described by singular families (see \cite{S}).

\begin{prop}\label{prop:from-mu-to-Amu}
For every set of vectors $S=\{\bmu_1\cdots\bmu_m\}$ in $\{0,1\}^n$ such that each vector $\bmu_j$ has at least two of its arguments different from zero, there is a matrix $A_S\in\Sig$ such that the meromorphic continuation of $\zeta_{A_S}$ has all its poles located at the hyperplanes $\lan\bmu_j,s\ran=0$ for every $1\leq j\leq m$.
\end{prop}

\begin{proof}
For a set $S$ as in the statement, consider the $n\times (m+1)$ matrix $A_{S}$ where $\bmu_j$ is the $j$-th column for $1\leq j\leq m$ and the last column is full of ones. We claim that $\zeta_{A_S}$ has poles at $\lan\bmu_j,s\ran=0$ for every $j$. Indeed, $\Delta_{\lan\bmu_j,\bep\ran}+\R^n_+$ is the intersection of the half spaces $\lan e_i,\bsigma\ran>0$ and $\lan\bmu_j,\bsigma\ran>1$, where $\{e_i\}_{i=1}^n$ is the canonical basis of $\R^n$. The result then follows from Theorem \ref{thm:PolarlocusShintani}.
\end{proof}

\begin{rk}
Notice that in the previous Proposition, the case where $\bmu$ is an element of the canonical basis $\{e_i\}_{i=1}^n$ is not considered. This is a consequence of Theorem \ref{thm:PolarlocusShintani} where the hyperplanes of the type $\lan e_i,\bs\ran=\nu_k^J-l$ can carry poles only when $0\leq l\leq\nu_k^J-1$. The reason for this is that $\zeta_A=\frac{\M_{\S(A)}}{\Gamma}$ (see Proposition \ref{prop:Shintani_abs_conv_domain}), and therefore the poles of the Gamma functions cancel those of $\M_{\S(A)}$ lying on the hyperplanes $\lan e_i,\bs\ran=l$ with $l\in\Z_{\leq0}$. 
\end{rk}

The following results aim to prove a statement converse to that on Proposition \ref{prop:from-mu-to-Amu}.
For this purpose, let us introduce some notation first. For any subset $I$ of $[n]$, we denote by $\pi_I:\R^n\to\R^{n-|I|}$ the projection orthogonal to the subspace $\bigoplus_{i\in I}\R(e_i)$. We also denote by $\iota_I:\R^{n-|I|}\to\R^n$ the injection map such that $Im(\iota_I)=\bigoplus_{i\in [n]\setminus I}\R(e_i)$ and \[\pi_I\circ\iota_I={\rm Id}_{\R^{n-|I|}}.\]
Notice also that, when restricted to $\bigoplus_{i\in[n]\setminus I}\R(e_i)\subset \R^n$, \begin{equation}\label{eq:rightinverserestricted}\iota_I\circ\pi_I\vert_{\bigoplus_{i\in[n]\setminus I}\R(e_i)}={\rm Id}_{\bigoplus_{i\in[n]\setminus I}\R(e_i)}.
\end{equation} 
For example, if $n=4$ and $I=\{2,4\}$, then $\pi_I((v_1,v_2,v_3,v_4))=(v_1,v_3)$, and $\iota_I((w_1,w_2))=(w_1,0,w_2,0)$.

\begin{lem}\label{lem:HyperplanesofNewPol} 
Let $D$ be a subset of $\R^n_+$ such that $D+\R_+^n=\iota_I(\pi_I(D+\R_+^n))+\sum_{i\in I}\R_+(e_i)$. If there is a $\hat\bmu\in\R^{n-|I|}$ such that \[\pi_I(D+\R_+^n)=\R^{n-|I|}_+\cap\{\tilde\bsigma\in\R^{n-|I|}: \lan\tilde\bsigma,\hat\bmu\ran>1\},\] then \begin{equation}\label{eq:hyperplaneInt}D+\R_+^n=\R^{n}_+\cap\{\bsigma\in\R^{n}: \lan\bsigma,\bmu\ran>1\},\end{equation}
where $\bmu=\iota_I(\hat\bmu)$, i.e, $\pi(\bmu)=\hat\bmu$ and $\mu_i=0$ for every $i\in I$.
\end{lem}

\begin{proof}
We prove the inclusion from left to right in \eqref{eq:hyperplaneInt}. For $\bsigma\in D+\R^n_+$, it clearly implies $\bsigma\in \R^n_+$. Notice that, for $\tilde\bsigma:=\pi_I(\bsigma)$, we have $\lan\tilde\bsigma,\hat\bmu\ran=\lan\bsigma,\bmu\ran$ since $\pi_I(\bmu)=\hat\bmu$ and $\mu_i=0$ for every $i\in I$. It follows that $\lan\bsigma,\bmu\ran>1$ and the inclusion is demonstrated.

We prove the inclusion from right to left in \eqref{eq:hyperplaneInt}. Let $\bsigma\in\R^n_+$ such that $\lan\bsigma,\bmu\ran>1$, we claim $\bsigma\in\iota_I(\pi_I(D+\R_+^n))+\sum_{i\in I}\R_+(e_i).$ Indeed, it follows from $\lan\tilde\bsigma,\hat\bmu\ran=\lan\bsigma,\bmu\ran$, and since $\bsigma\in\R^n_+$ its $i$-th  coordinate is positive for every $i\in I$. The result follows from the equality $D+\R_+^n=\iota_I(\pi_I(D+\R_+^n))+\sum_{i\in I}\R_+(e_i)$.
\end{proof}

Recall that for a column vector $C_j$ we denote, with some abuse of notation, by the same symbol the polynomial $C_j(\bep)=\lan\bep,C_j\ran$, and then its Newton polytope $\Delta_{C_j}$ is well defined.
 
\begin{lem}\label{lem:poles-1-column}
Let $A$ be a matrix in $\Sig$, with $C_j$ the columns of $A$. For each $j\in[r]$, consider the vector $\bmu_j\in \{0,1\}^n$ with its $i$-th coordinate equal to zero if, and only if, the $i$-th element of the column $C_j$ is a zero. Then 
\begin{equation}\label{eq:muj-for-j-unitary}
\Delta_{C_j}+\R^n_+=\bigcap_{i\in [n]}\{\bsigma\in\R^{n}: \lan\bsigma,e_i\ran>0\}\cap\{\bsigma\in\R^{n}: \lan\bsigma,\bmu\ran>1\}=\R^n_+\cap\{\bsigma\in\R^{n}: \lan\bsigma,\bmu_j\ran>1\}.
\end{equation}
\end{lem}
\begin{rk}
Notice that the first set of intersections in middle term in equation \eqref{eq:muj-for-j-unitary} amount to $\R^n_+$, and therefore the middle and last terms are trivially equal. This formulation will be useful for Lemma \ref{lem:poles-several-comluns-inclusion}. 
\end{rk}
\begin{proof}
We analyze two different cases:
\begin{enumerate}
\item Assume that $C_j(\bep)=\sum_{i=1}^n a_{ij}\epsilon_j$ is such that $a_{ij}\neq0$ for every $i$. Then 
\[\Delta_{C_j}=\Big\{\bsigma=\sum_{i\in[n]}\lambda_i e_i\in\R^n_+:\sum_{i\in n}\lambda_i=1\Big\}\]is a $n-1$ dimensional simplex. Setting $\bmu=(\underbrace{1,\cdots,1}_{n-{\rm times}})$ is follows that
\[\Delta_{C_j}+\R^n_+=\R^n_+\cap\{\bsigma\in\R^{n}: \lan\bsigma,\bmu_j\ran>1\},\]
and the result follows.

\item Consider now the case where $a_{ij}=0$ for $i\in I_j$ where $I_j\subset[n]$ and $I_j\neq[n]$. It implies that $\Delta_{C_j}=\iota_{I_j}(\pi_{I_j}(\Delta_{C_j}))$. In words it means that $\Delta_{C_j}$ is contained in the subspace $\bigoplus_{i\in[n]\setminus I_j}\R_+(e_i)\subset\R^n$, and thus \begin{equation}\label{eq:2ndhypforlastlem}\Delta_{C_j}+\R^n_+=\Delta_{C_j}+\sum_{i\in[n]\setminus I_j}\R_+(e_i)+\sum_{i\in I_j}\R_+(e_i)=\iota_{I_j}(\pi_{I_j}(\Delta_{C_j}+\R^n_+))+\sum_{i\in I_j}\R_+(e_i),
\end{equation} 
where we used \eqref{eq:rightinverserestricted} and the fact that $\sum_{i\in I_j}\R_+(e_i)\subset \ker(\pi_{I_j})$. Notice also that $\pi_{I_j}(\Delta_{C_j})$ is the Newton polytope of $\R^{n-|I_j|}\ni\bep\mapsto p(\bep)=\sum_{i=1}^{n-|I_j|}\epsilon_i$, then the previous case implies that \begin{equation}\label{eq:1sthypforlastlem}\pi_{I_j}(\Delta_{C_j})+\R^{n-|I_j|}_+=\R^{n-|I_j|}_+\cap \{\bsigma\in\R^{n-|I_j|}: \lan\bsigma,\hat\bmu\ran>1\}\end{equation} with $\hat\bmu=(\underbrace{1,\cdots,1}_{n-|I_j|\, {\rm times}})$. The result follows from \eqref{eq:2ndhypforlastlem}, \eqref{eq:1sthypforlastlem} and Lemma \ref{lem:HyperplanesofNewPol} with $D=\Delta_{C_j}$. 
\end{enumerate}

The case where all $a_{ij}=0$ is not considered since $A\in\Sig$ implies that each line and each row has at least an argument different from zero (Definition \ref{defn:Ashintani}).
\end{proof}

We proceed to generalize Lemma \ref{lem:poles-1-column} for $\Delta_{C_J}+\R^n_+$ where $|J|>1$. To that end we introduce some notations: Let $\bmu\in\{0,1\}^n$, we call the {\bf support} of $\bmu$, denoted by ${\rm Supp}(\bmu)$, the subset of $[n]$ given by \[{\rm Supp}(\bmu):=\{i\in[n]:\mu_i\neq0\}.\]
Notice that for every subset $S$ of $[n]$ there is a unique vector $\bmu\in\{0,1\}^n$ such that ${\rm Supp}(\bmu)=S$. 
For $A$ a matrix in $\Sig$, and $J\subset[r]$, we denote by $\bmu_J$ the only vector in $\{0,1\}^n$ which satisfies ${\rm Supp}(\bmu_J)=\bigcup_{j\in J}{\rm Supp}(\mu_j)$. In particular if $J=\{j\}$, then $\bmu_j=\bmu_{\{j\}}$. We also write $S_j$ short for ${\rm Supp}(\bmu_j)$ and $S_J$ short for ${\rm Supp}(\bmu_J)$.

With the notations previously introduced, Lemma \ref{lem:poles-1-column} implies that
\begin{equation}\label{eq:Supp-Sum-unitary}\bsigma\in \Delta_{C_j}+\R^n_+\Leftrightarrow \bsigma\in\R^n_+\,\wedge\,\sum_{i\in S_j}(\bsigma)_i>1,
\end{equation}
where $(\bsigma)_i$ denotes the $i$-th coordinate of the vector $\bsigma$.

\begin{lem}\label{lem:poles-several-comluns-inclusion}
Let $A$ be a matrix in $\Sig$ and $J\subset[r]$. Then
\[\Delta_{C_J}+\R^n_+\subset\Big(\bigcap_{K\subset J\,|\, K\neq\emptyset}\{\bsigma\in\R^n:\lan\bsigma,\bmu_K\ran>|K|\}\Big)\cap\R^n_+,\]
where $C_J(\bep)=\prod_{j\in J}\lan\bep,C_j\ran$ as before.
\end{lem}

\begin{proof}
Using the fact that the Newton polytope of a product of polynomials is equal to the Minkowski sum of the Newton polytopes of each of the polynomials, one sees that $\Delta_{C_J}=\sum_{j\in J}\Delta_{C_j}$. Furthermore, since $\R^n_+=\R^n_++\R^n_+$, we have $\Delta_{C_J}+\R^n_+=\sum_{j\in J}\big(\Delta_{C_j}+\R^n_+\big)$. We prove then the following inclusion which is equivalent to the one in the statement of the Lemma.
\begin{equation*}
\sum_{j\in J}\big(\Delta_{C_j}+\R^n_+\big)\subset\Big(\bigcap_{K\subset J\,|\, K\neq\emptyset}\{\bsigma\in\R^n:\lan\bsigma,\bmu_K\ran>|K|\}\Big)\cap\R^n_+.
\end{equation*}

For each $j\in J$, consider $\bsigma_j\in\Delta_{C_j}+\R^n_+$. It is clear that $\sum_{j\in J}\bsigma_j\in \R^n_+$ since every $\Delta_{C_j}\subset\R^n_+$. On the other hand, for $K\subset J$ with $K\neq\emptyset$, \eqref{eq:Supp-Sum-unitary} implies that 
 \[\lan\bsigma_j,\bmu_K\ran=\sum_{i\in S_K}(\bsigma_j)_i>1\] 
if $j\in K$.  Otherwise \[\lan\bsigma_j,\bmu_K\ran>0\] since $\bsigma_j\in\R^n_+$. The last two inequalities imply that $\lan\sum_{j\in J}\bsigma_j,\bmu_K\ran>|K|$, and thus $\sum_{j\in J}\bsigma_j\in\big(\bigcap_{K\subset J,\,K\neq\emptyset}\{\bsigma\in\R^n:\lan\bsigma,\bmu_K\ran>|K|\}\big)\cap\R^n_+$, which completes the inclusion.

\end{proof}

The next Proposition states that the inclusion in Lemma \ref{lem:poles-several-comluns-inclusion} is actually an equality. Its proof borrows tools from graph theory which will be discussed in Section \ref{sec:algorithm}.

\begin{prop}\label{prop:polesasFeynman}
Let $A$ be a matrix in $\Sig$ and $J\subset [r]$. Then
\begin{equation*}\label{eq:proofFeynmanpoles}\Delta_{C_J}+\R^n_+=\Big(\bigcap_{K\subset J\,|\, K\neq\emptyset}\{\bsigma\in\R^n:\lan\bsigma,\bmu_K\ran>|K|\}\Big)\cap\R^n_+.\end{equation*}
\end{prop}

\begin{proof}
The inclusion from left to right was proven in Lemma \ref{lem:poles-several-comluns-inclusion}. For he other direction we prove the equivalent inclusion 
\begin{equation}\label{eq:auxiliary-for-lemma}
\sum_{j\in J}\big(\Delta_{C_j}+\R^n_+\big)\supset\Big(\bigcap_{K\subset J\,|\, K\neq\emptyset}\{\bsigma\in\R^n:\lan\bsigma,\bmu_K\ran>|K|\}\Big)\cap\R^n_+,
\end{equation}
where the $\bmu_j$s are as in Lemma \ref{lem:poles-1-column}. Notice also that for an element $\bsigma\in\R^n_+$ which lies in the set on the right hand side of \eqref{eq:auxiliary-for-lemma}, the equivalence \eqref{eq:Supp-Sum-unitary} implies that for every $K\subset J$ with $K\neq\emptyset$
\begin{equation*}\sum_{i\in S_K}(\bsigma)_i>|K|.
\end{equation*}
It follows from Corollary \ref{cor:algorithm} below that for every $j\in[n]$ there is $\bsigma_j\in \R^n_{+}$, such that \begin{equation*}
\sum_{i\in S_j}(\bsigma_j)_i>1,
\end{equation*}
and $\bsigma=\sum_{j\in [n]}\bsigma_j$. The latter together with \eqref{eq:Supp-Sum-unitary} imply that $\bsigma$ lies in the left hand side of \eqref{eq:auxiliary-for-lemma}, which yields the result. 
\end{proof}

We now state the main Theorem of this section which provides a better description of the type of vectors $\bmu_J^K$ that parametrise the poles of a Shintani zeta function according to Theorem \ref{thm:PolarlocusShintani}.

\begin{thm}\label{thm:Feynman-mu}
Let $A$ be a matrix in $\Sig$, with $\{C_j\}_{j\in[r]}$ the set of columns of $A$, and write $C_J(\bep):=\prod_{j\in J}\lan \bep,C_j\ran$ for every $J\subset[r] $.  Consider also for every $J\subset[r]\setminus\emptyset$ the vector $\bmu_J\in\{0,1\}^n$ which is the only one satisfying ${\rm Supp}(\bmu_J)=\bigcup_{j\in J}{\rm Supp}(\bmu_j)$. Then the possible singularities of the meromorphic extension of $\zeta_A$ are located on the hyperplanes 
\[\lan \bmu_J,\bs\ran=|J|-l,\] where $l\in\Z_{\geq0}$. If moreover $\bmu_J$ is a vector of the canonical basis of $\R^n$, then $l$ only take values on the set $\{1,2,\cdots,|J|-1\}$.
\end{thm}
\begin{proof}
The vectors $\boldsymbol\mu_k^J$ from Theorem \ref{thm:PolarlocusShintani} are the normal vectors to the facets of the polyhedra $\Delta_{C_J}+\R^n_+$ with integer, mutually coprime coefficients. These vectors correspond to the $\boldsymbol\mu_K$s from Proposition \ref{prop:polesasFeynman}. The result follows from Proposition \ref{prop:polesasFeynman}.
\end{proof}

We proceed to describe the family of germs of meromorphic functions at zero, spanned by the germs containing the Shintani zeta functions. Recall that a germ of meromorphic functions at zero is an equivalence class of meromorphic functions determined by the following relation \[f\sim g\Leftrightarrow (\exists\, \mathcal O {\rm \ open}): 0\in\mathcal O\,\wedge\,f|_{\mathcal{O}}=g|_{\mathcal{O}}. \]

Let $M_{\zeta,n}$ be the family of germs of meromorphic functions at zero spanned by those containing Shintani zeta functions of matrices with $n$ rows. In particular the germs in $M_{\zeta,n}$ depend on $n$ complex variables. Using the natural embeddings $\iota_{n,m}:M_{\zeta,n}\to M_{\zeta,m}$ whenever $n<m$, we define the direct limit $M_{\zeta}:=\underset{\longrightarrow}{{\rm Lim}}\,  M_{\zeta,n}$. 

Using the formalism developed in \cite{GPZ4}, and given that the germs in $M_{\zeta}$ have linear poles as it was proven by Matsumoto \cite{M1} and in Theorem \ref{thm:PolarlocusShintani}, we can describe a Laurent expansion of $[\zeta_A]\in M_\zeta$ for any $A\in\Sig$. Indeed, the germ $[\zeta_A]$ can be written as
\begin{equation*}
[\zeta_A]=\sum_{i\in I} S_i+h, 
\end{equation*} 
where $h$ is an holomorphic germ at zero and the $S_i$ are fractions of the form \begin{equation*}
\bs\mapsto S_i(\bs)=\frac{h_i(\bs)}{\prod\lan\boldsymbol\mu_k^J,\bs\ran},
\end{equation*}

where $h_i$ are holomorphic germs depending on variables orthogonal to their respective denominator, and the $\boldsymbol\mu_k^J$ are as in Theorem \ref{thm:PolarlocusShintani}.

As a direct consequence of the previous discussion and of Theorem \ref{thm:Feynman-mu}, we have the following result.

\begin{prop}
The germs on $M_\zeta$ are of the type $\sum_{i\in I} S_i+h$, where $h$ is an holomorphic germ and the $S_i$ are fractions of the form \[\bs\mapsto S_i(\bs)=\frac{h_i(\bs)}{\prod_{J\in\mathcal J_i}(\sum_{j\in J} s_j)},\]
where $\mathcal J_i$ is a finite collection of subsets of $[n]$ with more than one element.
\end{prop}

We finally give some examples of how to calculate the domain of absolute convergence and the possible singularities of a Shintani zeta function $\zeta_A$ from the matrix $A$. By means of Corollary \ref{cor:A-squeleton} we will only consider matrices with arguments equal to $0$ or $1$.

\begin{ex}
Consider the matrix 
\begin{equation*}
B=\begin{pmatrix}
1&1\\
1&1
\end{pmatrix}.
\end{equation*}
We calculate the vectors $\bmu_1=(1,1)=\bmu_2$ from Lemma \ref{lem:poles-1-column} related to the first and second columns of the matrix $A$. It follows that $\bmu_{\{1,2\}}=(1,1)$. By means of Propositions \ref{prop:Shintani_conv_domain} and \ref{prop:polesasFeynman} the sum $\zeta_A$ is absolutely convergent whenever $\bsigma=\Re\bs$ satisfies \[\sigma_1+\sigma_2>2.\]
By means of Theorem \ref{thm:Feynman-mu} the possible singularities of $\zeta_A$ are located on the hyperplanes \[s_1+s_2=2-l,\]
where $l$ takes values in $\Z_{\geq0}.$
\end{ex}

\begin{ex}
Consider the matrix 
\begin{equation*}
A=\begin{pmatrix}
1&0&0\\
1&1&1\\
0&1&0
\end{pmatrix}.
\end{equation*}
The vectors $\mu_J$ are as follows \[\bmu_1=(1,1,0),\quad \bmu_2=(0,1,1),\quad \bmu_3=(0,1,0),\quad \bmu_{\{1,2\}}=(1,1,1),\]\[\bmu_{\{1,3\}}=(1,1,0),\quad\bmu_{\{2,3\}}=(0,1,1),\quad{\rm and }\quad\bmu_{\{1,2,3\}}=(1,1,1).\] 
It follows from Propositions \ref{prop:Shintani_conv_domain} and \ref{prop:polesasFeynman} that the sum $\zeta_A$ is absolutely convergent whenever the following inequalities are satisfied
\[\sigma_2>1,\quad\sigma_1+\sigma_2> 2,\quad \sigma_2+\sigma_3\geq2,\quad{\rm and}\quad \sigma_1+\sigma_2+\sigma_3\geq3.\]
By means of Theorem \ref{thm:Feynman-mu} the possible singularities of $\zeta_A$ are located on the hyperplanes \[s_2=1,\quad s_1+s_2= 2-l,\quad s_2+s_3=2-l,\quad{\rm and}\quad s_1+s_2+s_3=3-l\]where $l$ takes values in $\Z_{\geq0}$.
\end{ex}

\section{Distributing weight over a graph}\label{sec:algorithm}

Interestingly the proof of Theorem \ref{thm:Feynman-mu}, more precisely the proof of Proposition \ref{prop:polesasFeynman}, borrows tools from graph theory, which is the reason why in this section we do an excursion to this theory. Our main objective is to prove Corollary \ref{cor:algorithm} which is essential for the proof of Proposition \ref{prop:polesasFeynman}. For an introduction to graph theory we refer the reader to one of the many good references in this subject, for instance \cite{MM}.
The main result of this section is Theorem \ref{thm:algorithm}, the proof of which provides an algorithm to "distribute" weight over an intersection graph of a family of sets, such that the weight at each vertex is never lower than an imposed bound. This algorithm is, to the author's knowledge, new.

{\bf Notation:} Throughout this section we use the round brackets to refer to the coordinate of a vector in the canonical basis of $\R^n$. More precisely, for $\bsigma\in\R^n$, $(\bsigma)_k$ is the $k$-th component of $\bsigma$ in the canonical basis of $\R^n$.

\

It is easy to check that any real number $\sigma\geq m$, with $m\in \Z_{\geq0}$, can be expressed as a sum $\sigma=\sum_{j=1}^m\sigma_j$ where each $\sigma_j$ is bigger or equal than $1$. A way of generalizing the previous statement is to consider $\bsigma$ in $\R^2_{\geq 0}$, such that $(\bsigma)_1\geq1$, $(\bsigma)_2\geq 1$ and $(\bsigma)_1+(\bsigma)_2\geq 3$. One can check that $\bsigma$ admits a splitting of the form $\bsigma=\sum_{j=1}^3\bsigma_j$, with the vectors $\bsigma_j$ lying in $\R^2_{\geq0}$, and such that $(\bsigma_1)_1\geq1$, $(\bsigma_2)_2\geq1$ and $(\bsigma_3)_1+(\bsigma_3)_2\geq1$.  Indeed, one possible solution is to set $\bsigma_1=(1,0)$, $\bsigma_2=(0,1)$ and $\bsigma_3=\bsigma-\bsigma_1-\bsigma_2$. The following theorem is a generalization of this fact to any finite dimension $n$ and any number of vectors $m$.
 
For $m$ and $n$ positive integers, consider an application which associates to every $j$ in $[m]$ a set $S_j\subset[n]$.

\begin{thm}\label{thm:algorithm}
Any $\bsigma$ in $\R^n_{\geq0}$, such that for every $K\subset [m]$
\begin{equation*}\label{eq:hyp-algor-thm}\sum_{i\in\bigcup_{k\in K}S_k}(\bsigma)_i\geq|K|,\end{equation*}
can be written as a sum $\bsigma=\sum_{j\in [m]}\bsigma_j$, where the vectors $\bsigma_j$ lie in $\R^n_{\geq0}$, and such that 
\begin{equation}\label{eq:thesis-algor-thm}\sum_{i\in S_j}(\bsigma_j)_i\geq1.\end{equation}
\end{thm}

\begin{rk}\label{rk:charac-matrix}
Notice that the family of sets $S_j$ used in Theorem \ref{thm:algorithm} can be represented by a matrix $\mathfrak A$ with $n$ rows and $m$ columns, built using the characteristic functions of the sets $S_j$. More precisely, if $\chi_j$ is the characteristic function of the set $S_j$, then the argument on the $i$-th row and $j$-th column of $\mathfrak A$ is $\chi_j(i)$. 

In particular, if in Proposition \ref{prop:polesasFeynman} $J=[r]$, and the sets $S_j$ are the supports of the vectors $\bmu_j$, then $\mathfrak{A}$ would be an $n\times r$ matrix with zeros in the same positions than the matrix $A$, and ones in the rest of the arguments. It follows from Corollary \ref{cor:A-squeleton} that $\zeta_A$ and $\zeta_{\mathfrak A}$ have the same polar structure.
\end{rk}

We do the proof of the previous theorem by induction over $m$, using an algorithm which uses the intersection graph of the sets $S_j$. To illustrate how it works, we present an example before giving the actual proof.

\begin{ex} Set $n=6$, $m=5$ and consider the sets $S_1=\{2,3\}$, $S_2=\{1,4,6\}$, $S_3=\{1,5\}$, $S_4=\{6\}$ and $S_5=\{3,4,5\}$. The matrix built with the characteristic functions of the sets $S_j$ as in Remark \ref{rk:charac-matrix} is
\[\mathfrak A=\begin{pmatrix}
0&1&1&0&0\\
1&0&0&0&0\\
1&0&0&0&1\\
0&1&0&0&1\\
0&0&1&0&1\\
0&1&0&1&0
\end{pmatrix}\]

Consider the vector $\bsigma=(1.9, 0.6, 0.6, 0.8, 0.2,1.4)$. One can check that it satisfies
\[\sum_{i\in\bigcup_{k\in K}S_k}(\bsigma)_i\geq|K|\]for every $K\subset [5]$ with $K\neq\emptyset$.

Since our proof is by induction, consider the vectors $\bsigma_1=(0,0.6,0.6,0,0,0)$, $\bsigma_2=(0.1,0,0,0.8,0,0.3)$, $\bsigma_3=(1.8,0,0,0,0.2,0)$ and $\bsigma_4=(0,0,0,0,0,1.1)$, and therefore \[\bsigma=\sum_{j=1}^4\bsigma_{j}.\] Notice that for every $j\in[4]$, the vector $\bsigma_{j}$ satisfies $\sum_{i\in S_j}(\bsigma_j)_i\geq1$. We proceed to describe an algorithm to build another vector $\bsigma_5$ out of the current vectors $\bsigma_j$ with $1\leq j\leq4$ such that it satisfies the conditions on the statement of Theorem \ref{thm:algorithm}.

Build the intersection graph $\mathcal G$ of the family of sets $S_j$ with $j\in [5]$ (see for instance \cite{MM}), i.e. a non oriented graph whose set of vertices is $[5]$ and whose set of edges is \[E(\mathcal G):=\{(i,j)\in[5]^2: S_i\cap S_j\neq\emptyset \}.\]
Namely,

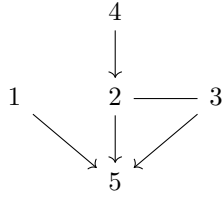
\begin{equation*}
\begin{tikzcd}
&   4\arrow[d] \\
1\arrow[rd]&2\arrow[r,dash]\arrow[d]&3\arrow[ld]\\
& 5& \\
\end{tikzcd} 
\end{equation*}

Since at this point, the only coordinates of the vectors $\bsigma_j$ which are bigger than zero are the coordinates in $S_j$, this graph represents the way the vectors can "share" some amount to the vector $\bsigma_5$. This is represented by the arrows, but keep in mind that the graph is not oriented.

Take for instance the vertex 1. It has an edge connecting to the vertex 5. This means $S_1\cap S_5\neq\emptyset$. Indeed $\{3\}\subset S_1\cap S_5$, and therefore we can subtract an amount $\lambda$ from the third coordinate of $\bsigma_1$ and add it to $\bsigma_5$. This value $\lambda$ cannot be bigger than $0.6$ since $(\bsigma_1)_3=0.6$ and then $\bsigma_1$ would leave $\R^6_+$. It also cannot be bigger than $0.2$ since $\sum_{i\in S_1}(\bsigma_1)_i=0.6+0.6$ must remain bigger than $1$. We therefore set $\lambda=0.2$, and set the new vectors \[\bsigma_1=(0,0.6,0.4,0,0,0)\]and
\[\bsigma_5=(0,0,0.2,0,0,0).\]

Proceed to the vertex 2, which is still connected by an edge to the vertex 5 since $4\in S_2\cap S_5$. By a similar analysis as before, we realize we can subtract $\lambda=0.2$ again tho the fourth coordinate of $\bsigma_2$ and add it to $\bsigma_5$, having then\[\bsigma_2=(0.1,0,0,0.6,0,0.3),\]and \[\bsigma_5=(0,0,0.2,0.2,0,0).\] 
We proceed to the vertex 3 and by the same analysis we subtract $\lambda=0.2$ to the fifth coordinate of $\bsigma_3$ and add it to the fifth coordinate of $\bsigma_5$ setting \[\bsigma_3=(1.8,0,0,0,0,0),\]and \[\bsigma_5=(0,0,0.2,0.2,0.2,0).\]
However this time, even with $\sum_{i\in S_3}(\bsigma_3)_i=1.8$ the vector $\bsigma_3$ cannot share anything else with $\bsigma_5$ since $\sum_{i\in S_3\cap S_5}(\bsigma_3)_i=0$. For this reason we cut the vertex between vertices 3 and 5. The new graph $\mathcal G$ is therefore
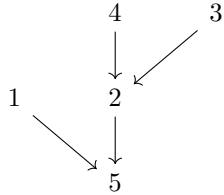
\begin{equation*}
\begin{tikzcd}
& 4\arrow[d]& 3\arrow[ld]\\
1\arrow[rd]&2\arrow[d]&\\
&5& \\
\end{tikzcd} 
\end{equation*} 
and vertex 3 is now at least at 2 steps from vertex 5.

We proceed now to analyze vertices which are not direct neighbors of 5. Consider the vertex 4: by a similar argumentation as before the sixth coordinate of $\bsigma_4$ can share a maximum of $0.1$ to $\bsigma_2$, which then can shares $0.1$ of its fourth coordinate with $\bsigma_5$. Then the new vectors are\[\bsigma_4=(0,0,0,0,0,1),\]
\[\bsigma_2=(0.1,0,0,0.5,0,0.4)\]and \[\bsigma_5=(0,0,0.2,0.3,0.2,0).\] 

We finally revisit vertex 3 but now as an indirect neighbor of the vertex 5. We see now that we may subtract $0.8$ from the first component of $\bsigma_3$ and add it to $\bsigma_2$. Then $\bsigma_2$ can shares a maximum of $0.5$ from its fourth component with $\bsigma_5$, and then
\[\bsigma_3=(1,0,0,0,0,0),\]
\[\bsigma_2=(0.9,0,0,0,0,0.4),\]and \[\bsigma_5=(0,0,0.2,0.8,0.2,0).\] Since $\sum_{i\in S_2\cap S_5}(\bsigma_2)_i=0$ we cut the edge between the the vertices 2 and 5 and the graph now looks like this.
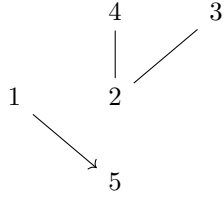
\begin{equation*}
\begin{tikzcd}
&  4\arrow[d, dash]& 3\arrow[ld,dash]\\
1\arrow[rd]&2&\\
& 5& \\
\end{tikzcd} 
\end{equation*} 
The vectors we end up with are:\begin{align*}
\bsigma_1&=(0,0.6,0.4,0,0,0)\\
\bsigma_2&=(0.9,0,0,0.5,0,0.4)\\
\bsigma_3&=(1,0,0,0,0,0)\\
\bsigma_4&=(0,0,0,0,0,1)\\
\bsigma_5&=(0,0,0.2,0.8,0.2,0),
\end{align*}
which clearly satisfy the requirements. Indeed $\bsigma\in\R^6_+$, $\sum_{i\in S_j}(\bsigma_j)_i\geq1$, and $\bsigma=\sum_{j=1}^5\bsigma_j$. Notice also that the final graph has two separated components. This means, as it can be checked, that the vectors corresponding to vertices in the components not containing 5 (namely the vertices 2, 3, and 4) have no intersection in their supports with $S_5$, or the coordinates in the intersection are equal to $0$. The last implies that they cannot share any more with the vector $\bsigma_5$. 

\end{ex}

After having illustrated in the example how the algorithm works we shall provide the actual proof of Theorem \ref{thm:algorithm}

\begin{proof}[Proof of Theorem \ref{thm:algorithm}]

The proof is by induction over $m$. The case $m=1$ is trivial setting $\bsigma_1=\bsigma$. For the inductive step assume the statement is true for $m-1$ and consider a vector $\bsigma\in \R_{\geq0}^n$ which, for every $K\subset [m]$, satisfies 
\[\sum_{i\in\bigcup_{k\in K}S_k}(\bsigma)_i\geq|K|.\]

It follows from the induction hypothesis that there exist vectors $\bsigma_j\in\R^n_{\geq 0}$ for $1\leq j\leq m-1$ satisfying \eqref{eq:thesis-algor-thm}, and such that $\bsigma=\bsigma_{j_1}+\cdots+\bsigma_{j_{m-1}}$.

Define the intersection graph $\mathcal G$ for the sets $S_j$, namely the graph whose set of vertices is $[m]$ and whose set of edges is  \[E(\mathcal G):=\{(i,j)\in[m]^2: {\rm Supp}(\bmu_i)\cap{\rm Supp}(\bmu_j)\neq\emptyset \}.\] 
Consider then the distance map $d_m:[m]\to\Z_{\geq0}\cup\{\infty\}$, where $d_m(j)$ is the minimum number of steps one has to give in the graph $\mathcal G$ to go from the vertex $j$ to the vertex $m$. If there is no path from $j$ to $m$ we set $d_m(j)=\infty$. Notice that this distance map does not define an order on the connected component of the graph containing the vertex $m$, since the antisymmetry might fail.  For every $j\in [m-1]$ such that $d_m(j)\neq\infty$ define $S_{j,<}:=\{i\in S_j: (\exists l\in [m])\, i\in S_l\ \wedge\  d_m(l)=d_m(j)-1\}$. In words, $S_{j,<}$ is the subset of elements in $S_j$ which are also in some $S_l$, where $l$ is a vertex closer to $m$ in the the graph $\mathcal G$. 

Finally set a counter $\alpha=1$ which will help us analyze the vertices of $\mathcal G$ depending on their distance to the vertex m.

We use the notation $:\models$ to redefine a parameter on the left hand side using its old value on the right hand side. For instance $\alpha:\models \alpha+1$ means we add $1$ to the value of the counter $\alpha$. 

\begin{enumerate}
\item Take a vertex $j$  such that $d_m(j)=\alpha$, and let $\rho(j)$ be the set of paths of length $\alpha$ going from $j$ to $m$. 
\item Choose a path $\varrho=(h_1,\cdots,h_\alpha)\in \rho(j)$, this means that $h_1=j$ and $h_\alpha=m$, and set a counter $\beta=1$.
\item For the vertex $h_\beta$, consider \[r={\rm min}\{\sum_{i\in S_{h_\beta}}(\bsigma_{h_{\beta}})_i-1,\,\sum_{i\in S_{h_\beta}\cap S_{h_{\beta+1}}}(\bsigma_{h_\beta})_i\}\]
 and subtract from the coordinates $(\bsigma_{h_\beta})_i$ with $i\in S_{h_\beta}\cap S_{h_{\beta+1}}$ a total amount of $r$, which will be added to the same coordinates of the vector $\bsigma_{h_{\beta+1}}$. This is, for $i\in S_{h_\beta}\cap S_{h_{\beta+1}}$ redefine 
 \begin{align*}(\bsigma_{h_\beta})_i&:\models (\bsigma_{h_\beta})_i-\lambda_i,\ \ \ {\rm and}\\
 (\bsigma_{h_{\beta+1}})_i&:\models (\bsigma_{h_{\beta+1}})_i+\lambda_i
 \end{align*} where $0\leq\lambda_i<(\bsigma_{h_\beta})_i$, and such that $\sum_{i\in S_{h_\beta}\cap S_{h_{\beta+1}}}\lambda_i=r$.   

Notice at this point that the existence of the $\lambda_i$s is granted by the way $r$ is defined. Notice also that the new vectors $\bsigma_{h_\beta}$ and $\bsigma_{h_{\beta+1}}$ still satisfy \eqref{eq:thesis-algor-thm}, except if $h_{\beta+1}=m$.

If $\beta<\alpha-1$, add one to the counter $\beta$, namely $\beta:\models \beta+1$ and repeat step 3. If $\beta=\alpha-1$ continue with the next step.

\item We have three options at this point:
\begin{itemize}
\item If $\sum_{i\in S_{h_\beta,<}}(\bsigma_{h_\beta})_i>0$ and $\beta>1$, subtract $1$ to the value of $\beta$, namely
\[\beta:\models \beta-1,\]and repeat step 4.
\item If $\sum_{i\in S_{h_\beta,<}}(\bsigma_{h_\beta})_i=0$, it means that the vector $\bsigma_{h_\beta}$ has shared everything from the coordinates in $S_{h_\beta,<}$ . It implies that it cannot share any more to a vertex closer to $m$. In this case remove the edges $(h_\beta,h_{\beta+1})$ and $(h_{\beta+1},h_{\beta})$ from the set of edges of the graph $\mathcal G$, recalculate the values of the distance map $d_m$ for every vertex and go step 1, without resetting the value of $\alpha$.
\item If $\sum_{i\in S_{h_\beta,<}}(\bsigma_{h_\beta})_i>0$ and $\beta=1$, proceed to the next step.
\end{itemize}
\item We have again two options:
\begin{itemize}
\item If there is a path $\varrho\in\rho(j)$ that we haven't followed in step 2, choose that path and go back to step 2.
\item If the process from step 2, has been followed for every $\varrho\in\rho$, then proceed to the next step.
\end{itemize}
\item We have again three options:
\begin{itemize}
\item If there is any other vertex $j$ with $d_m(j)=\alpha$ that hasn't been chosen in step 1, choose that vertex $j$ and go back to step 1.
\item If the process in step 1 has been followed for every vertex $j$ such that $d_m(j)=\alpha$, and $\alpha$ is not maximal yet, namely $\alpha<{\rm max}\{d_m(j):j\in J\, \wedge\, d_m(j)<\infty\}$, then add 1 to the value of $\alpha$ \[\alpha:\models\alpha+1\]and go back to step 1.
\item  If the process from step 1 has been followed for every vertex $j$ such that $d_m(j)=\alpha$, and $\alpha$ is maximal, proceed to the last step.
\end{itemize}
\item For every $j$ in $[m-1]$, consider the coordinates $i\notin S_{j}$ and redefine \[(\bsigma_{j})_i:\models 0,\ \ \ {\rm and}\]\[(\bsigma_{m})_i:\models(\bsigma_{m})_i+(\bsigma_{j})_i,\]
and finish the algorithm.
\end{enumerate}

The algorithm is clearly finite since the graph $\mathcal G$ is finite.
We proceed to prove that the vectors $\bsigma_{j}$ obtained after  following the algorithm satisfy the requirements\begin{equation}\label{eq:2nd-requi-algorithm}\bsigma=\sum_{j=1}^m\bsigma_{j},\ \ \ \bsigma_{j}\in\R^n_{\geq0},\ \ \ {\rm and}\ \ \ \sum_{i\in S_{j}}(\bsigma_{j})_i\geq 1,\end{equation}
for every $j$ in $[m]$.

The first condition is satisfied by assumption and the fact that every amount that was subtracted from a coordinate of a vector, was added to the same coordinate of the other vector. The second condition is a consequence of having added positive values to the coordinates of the vectors. Only in step 3 a positive real number $\lambda_i$ is subtracted from the coordinates in the the support of the vector $\bsigma_{h_\beta}$, but $\lambda_i\leq(\bsigma_{h_\beta})_i$ and therefore all coordinates remain non negative. The fact that \eqref{eq:2nd-requi-algorithm} is satisfied for all $l\in[m-1]$ is also a consequence of the way $r$ is chosen in step 3, since $r<\sum_{i\in S_{h_\beta}}(\bsigma_{h_{\beta}})_i-1$, then the vector $\bsigma_{h_\beta}$ still satisfies \eqref{eq:2nd-requi-algorithm}. Finally we have to check that 
\[\sum_{i\in S_{m}}(\bsigma_{m})_i\geq1.\]

Notice first that for all vertices $j$ not in the connected component of the graph $\mathcal G$ containing the vertex $m$, the value of their coordinates $i\in S_{m}$ is equal to $0$. Let $K$ be the subset of $[m]$ whose vertices are on the connected component containing $m$. As a consequence of the algorithm, for every $k\in K$, $\sum_{i\in S_k}(\bsigma_k)_i=1$. Indeed, if it is still connected to the vertex m, step 4 indicates that $\sum_{i\in S_{k,<}}(\bsigma_k)_i> 0$.  This is only possible if at some point the third step of the algorithm gives $r=\sum_{i\in S_k}(\bsigma_k)_i-1$, which means that at the end of the algorithm $\bsigma_k$ is such that \[\sum_{i\in S_k}(\bsigma_k)_i=1.\] On the other hand, by assumption
\begin{align*}|K|\leq&\sum_{i\in \bigcup_{k\in K}S_k}(\bsigma)_i\\
=&\sum_{k\in K\setminus\{m\}}\big(\sum_{i\in S_k}(\bsigma_k)_i\big)+\sum_{i\in S_m}(\bsigma_{m})_i.
\end{align*}
The first term on the last line is equal to $(|K|-1)$, which implies that $\sum_{i\in S_m}(\bsigma_{j_m})_i\geq1$ as expected. This finishes the proof.

\end{proof} 

\begin{rk}
Notice that the process described in the proof is not an algorithm in a strict sense since it is not completely deterministic. For instance, there is some freedom when choosing the order in which one analyzes the vertices with the same $d_m$ in step 1, or the different paths in step 2. Also there might be a freedom in the way the $\lambda_i$s are chosen in step 3. Different choices might lead to different constructions of the vectors $\bsigma_m$, but it doesn't matter because they all satisfy the requirements of the theorem.

We must also warn that this process is not optimal in computation time since many redundancies may occur. It is written with the only purpose of demonstrating the existence of a solution to the equation $\bsigma=\sum_{j=1}^m\bsigma_{j}$ satisfying the conditions $\sum_{i\in S_{j}}(\bsigma_{j})_i\geq1$ and $\bsigma_{j}\in\R^n_{\geq0}$, provided that  for every $K\subset [m]$
\begin{equation*}\sum_{i\in\bigcup_{k\in K}S_k}(\bsigma)_i\geq|K|.\end{equation*}
\end{rk}

\begin{rk}
Theorem \ref{thm:algorithm} is related to Hall's marriage theorem \cite{H} and the theory of optimal transport. In this direction Thierry Champion could give an alternative proof of the result using von Neumann's minmax theorem.
\end{rk}

We now prove that that Theorem \ref{thm:algorithm} is also true if the inequalities in the statement are strict. Recall our assumption that for $m$ and $n$ positive integers, we fix an application which associates to every $j$ in $[m]$ a set $S_j\subset[n]$.

\begin{cor}\label{cor:algorithm}
Any $\bsigma$ in $\R^n_{+}$, such that for every $K\subset [m]$

\begin{equation}\label{eq:hyp-algor-cor}\sum_{i\in\bigcup_{k\in K}S_k}(\bsigma)_i>|K|,\end{equation}
can be written as a sum $\bsigma=\sum_{j\in [m]}\bsigma_j$, where the vectors $\bsigma_j$ lie in $\R^n_{+}$, and such that \begin{equation}\label{eq:thesis-algor-cor}\sum_{i\in S_j}(\bsigma_j)_i>1.
\end{equation} 
\end{cor}
\begin{proof}
The result follows from Theorem \ref{thm:algorithm} and the fact that the set of $\bsigma\in \R^n_+$ satisfying \eqref{eq:hyp-algor-cor} is an open set. More precisely, for each $i\in [n]$ consider $\lambda_i>0$ small enough such that $\bsigma':=\bsigma-\boldsymbol\lambda$ still satisfies \eqref{eq:hyp-algor-cor}, and $\bsigma'\in \R^n_+$. By means of Theorem \ref{thm:algorithm} there are $\bsigma_j'\in \R^n_{\geq0}$ for $1\leq j\leq m$ satisfying \eqref{eq:thesis-algor-thm}, and such that $\bsigma'=\sum_{j=1}^m \bsigma_j'$. It follows that the vectors $\bsigma_j:=\bsigma_j+ \frac{1}{m}\boldsymbol\lambda\in\R^n_+$ satisfy \eqref{eq:thesis-algor-cor} and the result follows.
\end{proof}

\end{document}